\documentclass[reqno,english]{amsart}%
\usepackage{amsfonts,amsmath,latexsym,verbatim,amscd,mathrsfs,color,array}
\usepackage{amsmath,amssymb,amsthm,amsfonts}
\usepackage{graphicx,color}
\usepackage{cite}
\usepackage{graphicx}
\usepackage{amsmath}
\usepackage{amsfonts}
\usepackage{amssymb}%
\providecommand{\U}[1]{\protect\rule{.1in}{.1in}}

\newtheorem{theorem}{Theorem}
\newtheorem{acknowledgement}[theorem]{Acknowledgement}

\newtheorem{corollary}[theorem]{Corollary}

\newtheorem{lemma}[theorem]{Lemma}

\newtheorem{proposition}[theorem]{Proposition}
\newtheorem{remark}[theorem]{Remark}

\numberwithin{equation}{section}
\email{liuyong@ncepu.edu.cn}
\email{jcwei@math.ubc.ca}
\begin{document}
\title[Nondegeneracy of Toda lattice]{Nondegeneracy of the traveling lump solution to the $2+1$ Toda lattice}
\author[Y. Liu]{Yong Liu}
\address{\noindent School of Mathematics and Physics, North China Electric Power
University, Beijing, China}
\author[J. Wei]{Juncheng Wei}
\address{\noindent Department of Mathematics, University of British Columbia,
Vancouver, B.C., Canada, V6T 1Z2}

\begin{abstract}
We consider the $2+1$ Toda system
\[
\frac{1}{4}\Delta q_{n}=e^{q_{n-1}-q_{n}}-e^{q_{n}-q_{n+1}}\text{ in
}\mathbb{R}^{2},\ n\in\mathbb{Z}.
\]
It has a traveling wave type solution $\left\{  Q_{n}\right\}  $ satisfying
$Q_{n+1}(x,y)=Q_{n}(x+\frac{1}{2\sqrt{2}},y)$, and is explicitly given by
\[
Q_{n}\left(  x,y\right)  =\ln\frac{\frac{1}{4}+\left(  n-1+2\sqrt{2}x\right)
^{2}+4y^{2}}{\frac{1}{4}+\left(  n+2\sqrt{2}x\right)  ^{2}+4y^{2}}.
\]
In this paper we prove that \{$Q_{n}$\} is nondegenerate.

\end{abstract}
\maketitle

\section{Introduction and statement of main results}

Toda lattice equation is a classical integrable system appearing in various
different areas of mathematics, mechanics and physics. In this paper, we are
interested in the lump solution to the following $2+1$ Toda lattice equation:
\begin{equation}
\frac{1}{4}\Delta q_{n}=e^{q_{n-1}-q_{n}}-e^{q_{n}-q_{n+1}}\text{ in
}\mathbb{R}^{2},n\in\mathbb{Z}. \label{Toda}%
\end{equation}
Equation $\left(  \ref{Toda}\right)  $ has been studied in \cite{A2,A3,A4,A},
using the inverse scattering transform(IST). A family of lump solution to
$\left(  \ref{Toda}\right)  $ has been found in \cite{A2} (see equation (3.6)
there). Let us consider one of these lumps:
\begin{equation}
Q_{n}\left(  x,y\right)  =\ln\frac{\frac{1}{4}+\left(  n-1+2\sqrt{2}x\right)
^{2}+4y^{2}}{\frac{1}{4}+\left(  n+2\sqrt{2}x\right)  ^{2}+4y^{2}}.
\label{lump}%
\end{equation}
Then $Q_{n}$ decays at the rate $O\left(  r^{-1}\right)  $, as $r^{2}%
=x^{2}+y^{2}\rightarrow+\infty.$ We also point out that in \cite{N}, families
of rational and $N$-breather solutions to $\left(  \ref{Toda}\right)  ,$
including $Q_{n},$ have been found using Hirota's direct method. It turns out
that $Q_{n}$ is actually an analogy of the classical lump solution to the KP-I
equation. As a matter of fact, the KP-I equation can be regarded as a
continuum limit of a family of generalized Toda lattice, with $\left(
\ref{Toda}\right)  $ being in this family. We refer to \cite{P} for more
details on this correspondence.

It is worth noting that the hyperbolic version of $\left(  \ref{Toda}\right)
:$%
\begin{equation}
\frac{1}{4}\left(  \partial_{x}^{2}-\partial_{y}^{2}\right)  q_{n}%
=e^{q_{n-1}-q_{n}}-e^{q_{n}-q_{n+1}},\left(  x,y\right)  \in\mathbb{R}%
^{2},n\in\mathbb{Z}, \label{Hy}%
\end{equation}
has also been investigated in \cite{A2,A3,A4,A}. From the IST point of view,
$\left(  \ref{Hy}\right)  $ is quite different from $\left(  \ref{Toda}%
\right)  .$ More precisely, the associated Cauchy problem in the IST
formulation is well posed in $\left(  \ref{Hy}\right)  ,$ but ill-posed in
$\left(  \ref{Toda}\right)  .$

\medskip The system $\left(  \ref{Toda}\right)  $ is a generalization of
following $1+1$ Toda lattice%
\begin{equation}
\frac{1}{4}q_{n}^{\prime\prime}=e^{q_{n-1}-q_{n}}-e^{q_{n}-q_{n+1}},\text{
}n\in\mathbb{Z}\text{.} \label{one}%
\end{equation}
This is a classical integrable system. Compared to the $2+1$ Toda lattice
$\left(  \ref{Toda}\right)  ,$ the system $\left(  \ref{one}\right)  $ has
been extensively studied in the literature. We refer to \cite{Teschl,Toda} and
the reference therein for more discussion on this equation and related
topics.\medskip

It is worth mentioning that there is another class of Toda equation, which we
call finite Toda system (It is also called Toda molecule equation in
\cite{Hirota}):
\begin{equation}
\left(  \partial_{x}^{2}-\partial_{y}^{2}\right)  q_{n}=4e^{q_{n-1}-q_{n}%
}-4e^{q_{n}-q_{n+1}},\left(  x,y\right)  \in\mathbb{R}^{2},n\in1,...,N,
\label{Fi}%
\end{equation}
with $q_{0}=-\infty,q_{N+1}=+\infty.$ This is also an integrable system. The
elliptic version of (\ref{Fi}):
\begin{equation}
\left(  \partial_{x}^{2}+\partial_{y}^{2}\right)  q_{n}=4e^{q_{n-1}-q_{n}%
}-4e^{q_{n}-q_{n+1}},\left(  x,y\right)  \in\mathbb{R}^{2},n\in1,...,N,
\label{Fi2}%
\end{equation}
has been studied in \cite{Lin}, where classification and nondegeneracy of
solutions have been proved, by analyzing various explicit conserved quantities
of this system.

\medskip

One of the motivations of studying Toda system comes from the following
unexpected connection: the solutions of $\left(  \ref{Toda}\right)  ,\left(
\ref{one}\right)  ,\left(  \ref{Fi2}\right)  ,$ actually describe the
interface motion of the solutions of the Allen-Cahn equation
\[
-\Delta u=u-u^{3}.
\]
The general principle is the following: For each \textquotedblleft
regular\textquotedblright\ enough solution of the Toda system in
$\mathbb{R}^{1}$ or $\mathbb{R}^{2}$, one should be able to construct an
entire solution to the Allen-Cahn equation in $\mathbb{R}^{2}$ or
$\mathbb{R}^{3}$ whose nodal sets resemble the solutions to the Toda system.
This type of results has been obtained in \cite{ADW, M1,Liu}, using the method of
infinite dimensional Lyapunov-Schmidt reduction. The bounded domain case is
considered in \cite{DKW}. A key element in these constructions is the
nondegeneracy of solutions to the Toda system. See \cite{M1, Liu}.

\medskip

In this paper, we prove that $\left\{  Q_{n}\right\}  $ is nondegenerate. Our
main result is

\begin{theorem}
\label{Main}Let $\left\{  U_{n}\right\}  $ be a solution of the linearized
equation
\begin{equation}
\Delta U_{n}=e^{Q_{n-1}-Q_{n}}\left(  U_{n-1}-U_{n}\right)  -e^{Q_{n}-Q_{n+1}%
}\left(  U_{n}-U_{n+1}\right)  . \label{l}%
\end{equation}
Suppose $U_{n+1}\left(  x,y\right)  =U_{n}\left(  x+\frac{1}{2\sqrt{2}%
},y\right)  $ and
\[
U_{n}\left(  x,y\right)  \rightarrow0\text{ \ as }x^{2}+y^{2}\rightarrow
+\infty.
\]
Then
\[
U_{n}=c_{1}\partial_{x}Q_{n}+c_{2}\partial_{y}Q_{n},
\]
for some constants $c_{1},c_{2}.$
\end{theorem}

\medskip

Theorem \ref{Main}, combined with gluing arguments similar to those in  \cite{ADW, Liu}, yields the following result for Allen-Cahn equation

\begin{corollary}
The Allen-Cahn equation
\begin{equation}
-\Delta u= u-u^3 \ \ \ \ \ \ \ \mbox{in} \ {\mathbb R}^3
\end{equation}
has a family of singly periodic solutions whose zero level set $ \{ u =0 \}$ is approximately given by $ \cup_{n} \{ z= Q_n (x, y)\} $.

\end{corollary}

\medskip

The main idea of the proofs of Theorem \ref{Main}  is to consider the
B\"acklund transformation at the linearized level. More precisely we use the
linearized B\"{a}cklund transformation to transform a kernel $U_{n}$ of
$\left(  \ref{l}\right)  $ to a kernel of the linearized equation with respect
to the trivial solution, which is an operator of constant coefficient. This
type of arguments has been used in \cite{Mizumachi,Pego} for the analysis of
spectral property of some soliton solutions to $1+1$ Toda lattice and KdV
equation.
In this respect, we also refer to \cite{Mi2}, where the stability of
line solitons of the KP-II equation has been proved using Miura transformation. In \cite{Liu1} similar idea is used  to prove the  nondegeneracy of the lump solution
to the KP-I equation.

\medskip

\begin{acknowledgement}
The research of J. Wei is partially supported by NSERC of Canada. Y. Liu is
partially supported by the Fundamental Research Funds for the Central
Universities 13MS39.
\end{acknowledgement}

\section{\bigskip Preliminaries on the B\"acklund transformation of the 2+1
Toda lattice}

B\"acklund transformation has been used to study soliton solutions for many
integrable systems. We refer to \cite{Hirota,Rogers} for a general
introduction this topic.

In this paper, we use $D$ to denote the bilinear derivative operator. That
is,
\[
D_{s}^{m}D_{t}^{n}f\cdot g=\left[  \left(  \partial_{s}-\partial_{s^{\prime}%
}\right)  ^{m}\left(  \partial_{t}-\partial_{t^{\prime}}\right)  ^{n}\right]
\left(  f\left(  s,t\right)  g\left(  s^{\prime},t^{\prime}\right)  \right)
|_{s^{\prime}=s,t^{\prime}=t}.
\]
We already know that the lump $Q_{n}$ can be obtained via the inverse
scattering transform. It turns out that we can also find $Q_{n}$ by B\"acklund
transformation. Let us explain this in the sequel.

To use the form of the B\"acklund transformation as studied in \cite{Hirota},
we introduce the complex variables $s=x+iy,t=x-iy.$ Then $\Delta=4\partial
_{s}\partial_{t}.$ Setting $r_{n}=q_{n-1}-q_{n},$ we transform $\left(
\ref{Toda}\right)  $ into
\begin{equation}
\partial_{s}\partial_{t}r_{n}=e^{r_{n+1}}+e^{r_{n-1}}-2e^{r_{n}}%
,n\in\mathbb{Z}. \label{r}%
\end{equation}
Let us define $V_{n}$ by
\[
1+V_{n}=e^{r_{n}}.
\]
Equation $\left(  \ref{r}\right)  $ then becomes
\[
\partial_{s}\partial_{t}\ln\left(  1+V_{n}\right)  =V_{n+1}+V_{n-1}-2V_{n}.
\]
Introducing the so-called $\tau$-function $\tau_{n}$ by $V_{n}=\partial
_{s}\partial_{t}\ln\tau_{n},$ we get the following bilinear form for the Toda
lattice $\left(  \ref{r}\right)  $:
\begin{equation}
D_{s}D_{t}\tau_{n}\cdot\tau_{n}=2\left(  \tau_{n+1}\tau_{n-1}-\tau_{n}%
^{2}\right)  ,n\in\mathbb{N}. \label{so}%
\end{equation}

For $n\in\mathbb{N},$ we define
\begin{align*}
\kappa_{n}  &  =1,\\
\omega_{n}  &  =\sqrt{2}\left(  s+t\right)  +n+\left(  s-t\right)
+\frac{\sqrt{2}-1}{2},
\end{align*}
and
\begin{equation}
\theta_{n}=\left(  \sqrt{2}\left(  s+t\right)  +n\right)  ^{2}-\left(
s-t\right)  ^{2}+\frac{1}{4}. \label{th}%
\end{equation}
Then $\left\{  \kappa_{n}\right\}  ,\left\{  \omega_{n}\right\}  ,\left\{
\theta_{n}\right\}  $ are solutions of $\left(  \ref{so}\right)  .$ The lump
solution $Q_{n}$ is corresponding to $\theta_{n}.$

We are interested in the transformation from $\left\{  q_{n}\right\}  $ to
$\left\{  \tau_{n}\right\}  $ at the linearized level. Let us define the
linearized operator $T_{\theta}$ of $\left(  \ref{so}\right)  $:
\begin{align*}
\left(  T_{\theta}\eta\right)  _{n}  &  :=\partial_{s}\partial_{t}\eta
_{n}\theta_{n}-\partial_{s}\eta_{n}\partial_{t}\theta_{n}-\partial_{t}\eta
_{n}\partial_{s}\theta_{n}\\
&  +\eta_{n}\partial_{s}\partial_{t}\theta_{n}-\left(  \eta_{n+1}\theta
_{n-1}+\theta_{n+1}\eta_{n-1}-2\theta_{n}\eta_{n}\right)  .
\end{align*}
Similarly, we have $T_{\omega}$ and $T_{\kappa}$. Note that
\[
\left(  T_{\kappa}\eta\right)  _{n}=\frac{1}{4}\Delta\eta_{n}+2\eta_{n}%
-\eta_{n+1}-\eta_{n-1}.
\]

\begin{lemma}
\label{linear}Suppose $\left\{  U_{n}\right\}  $ satisfies the linearized
equation
\begin{equation}
\frac{1}{4}\Delta U_{n}=e^{Q_{n-1}-Q_{n}}\left(  U_{n-1}-U_{n}\right)
-e^{Q_{n}-Q_{n+1}}\left(  U_{n}-U_{n+1}\right)  ,n\in\mathbb{Z}. \label{fi}%
\end{equation}
Assume
\[
U_{n+1}\left(  x,y\right)  =U_{n}\left(  x+\frac{1}{2\sqrt{2}},y\right)
\]
and
\[
U_{n}\left(  x,y\right)  \rightarrow0,\text{ as }x^{2}+y^{2}\rightarrow
+\infty.
\]
Then the equation
\begin{equation}
\frac{1}{4}\Delta\tilde{\eta}_{n}=e^{Q_{n-1}-Q_{n}}\left(  U_{n-1}%
-U_{n}\right)  , \label{yi}%
\end{equation}
has a solution $\left\{  \tilde{\eta}_{n}\right\}  $ with $\tilde{\eta}%
_{n+1}\left(  x,y\right)  =\tilde{\eta}_{n}\left(  x+\frac{1}{2\sqrt{2}%
},y\right)  $ and
\[
\left\vert \tilde{\eta}_{0}\right\vert \leq\frac{C}{\sqrt{1+x^{2}+y^{2}}}.
\]
Moreover, $\eta_{n}:=\theta_{n}\tilde{\eta}_{n}$ solves the linearized
equation $T_{\theta}\eta=0.$
\end{lemma}

\begin{proof}
Let us denote the function $U_{n-1}-U_{n}$ by $v_{n}$. We deduce from $\left(
\ref{fi}\right)  $ that
\begin{equation}
\frac{1}{4}\Delta v_{n}=e^{Q_{n-2}-Q_{n-1}}v_{n-1}+e^{Q_{n}-Q_{n+1}}%
v_{n+1}-2e^{Q_{n-1}-Q_{n}}v_{n}. \label{rn}%
\end{equation}
Setting $f_{n}=\left(  e^{Q_{n-1}-Q_{n}}-1\right)  v_{n},$ we can write
$\left(  \ref{rn}\right)  $ as
\begin{equation}
\frac{1}{4}\Delta v_{n}-v_{n+1}-v_{n-1}+2v_{n}=f_{n+1}+f_{n-1}-2f_{n}.
\label{rf}%
\end{equation}
We use $\mathcal{F}\left(  f\right)  $ to denote the Fourier transform(in
$\mathbb{R}^{2}$) of the function $f.$ Taking Fourier transform in $\left(
\ref{rf}\right)  ,$ we obtain
\begin{equation}
\mathcal{F}\left(  v_{n}\right)  =\frac{-\left(  2-2\cos\frac{\pi\xi_{1}%
}{\sqrt{2}}\right)  \mathcal{F}\left(  f_{n}\right)  }{-\pi^{2}\left(  \xi
_{1}^{2}+\xi_{2}^{2}\right)  +2-2\cos\frac{\pi\xi_{1}}{\sqrt{2}}}. \label{v}%
\end{equation}
Observe that%
\[
-\pi^{2}\left(  \xi_{1}^{2}+\xi_{2}^{2}\right)  +2-2\cos\frac{\pi\xi_{1}%
}{\sqrt{2}}\leq0,
\]
and it equals zero if and only if $\xi_{1}=\xi_{2}=0.$ Then using $\left(
\ref{v}\right)  $ and the estimate
\[
Q_{n}-Q_{n+1}=O\left(  \left(  1+x^{2}+y^{2}\right)  ^{-1}\right)  ,
\]
we can show that
\[
v_{n}=O\left(  \left(  1+x^{2}+y^{2}\right)  ^{-1}\right)  .
\]
This implies
\[
f_{n}=O\left(  \left(  1+x^{2}+y^{2}\right)  ^{-2}\right)  .
\]

The equation $\left(  \ref{yi}\right)  $ has a solution $\tilde{\eta}_{n}$
with
\[
\mathcal{F}\left(  \tilde{\eta}_{n}\right)  =-\frac{1}{\pi^{2}\left(  \xi
_{1}^{2}+\xi_{2}^{2}\right)  }\mathcal{F}\left(  e^{Q_{n-1}-Q_{n}}%
v_{n}\right)  =-\frac{\mathcal{F}\left(  f_{n}+v_{n}\right)  }{\pi^{2}\left(
\xi_{1}^{2}+\xi_{2}^{2}\right)  }.
\]
In view of $\left(  \ref{v}\right)  ,$ we get
\begin{equation}
\mathcal{F}\left(  \tilde{\eta}_{n}\right)  =\frac{\mathcal{F}\left(
f_{n}\right)  }{-\pi^{2}\left(  \xi_{1}^{2}+\xi_{2}^{2}\right)  +2-2\cos
\frac{\pi\xi_{1}}{\sqrt{2}}}. \label{y}%
\end{equation}
On the other hand, using equation $\left(  \ref{fi}\right)  $, we obtain
\[
\int_{\mathbb{R}^{2}}f_{n}\left(  x,y\right)  dxdy=\int_{\mathbb{R}^{2}}%
v_{n}\left(  x,y\right)  dxdy=0.
\]
Hence by $\left(  \ref{y}\right)  ,$
\begin{equation}
\left\vert \tilde{\eta}_{0}\right\vert \leq\frac{C}{\sqrt{1+x^{2}+y^{2}}}.
\label{yitatilta}%
\end{equation}
The derivatives of $\tilde{\eta}_{0}$ can also be estimated.

Under the transformation
\[
r_{n}=q_{n-1}-q_{n},\text{ }V_{n}=e^{r_{n}}-1,\partial_{s}\partial_{t}\ln
\tau_{n}=V_{n},
\]
the original Toda lattice $\left(  \ref{Toda}\right)  $ becomes
\begin{equation}
\partial_{s}\partial_{t}\ln\left(  \frac{\partial_{s}\partial_{t}\tau_{n}%
\tau_{n}-\partial_{s}\tau_{n}\partial_{t}\tau_{n}+\tau_{n}^{2}}{\tau_{n}^{2}%
}\right)  =\partial_{s}\partial_{t}\ln\frac{\tau_{n+1}\tau_{n-1}}{\tau_{n}%
^{2}}. \label{e}%
\end{equation}
Linearizing these relations at $\theta_{n}$, we find that the function
$\eta_{n}=\tilde{\eta}_{n}\theta_{n}$ satisfies
\[
\partial_{s}\partial_{t}\frac{\left(  T_{\theta}\eta\right)  _{n}}%
{\theta_{n+1}\theta_{n-1}}=0.
\]
This together with the estimate $\left(  \ref{yitatilta}\right)  $ tells us
$T_{\theta}\eta=0.$\bigskip
\end{proof}

Let
\[
\mathcal{P}=\left[  D_{t}D_{s}\tau_{n}\cdot\tau_{n}-2\tau_{n+1}\tau
_{n-1}+2\tau_{n}^{2}\right]  \tau_{n}^{\prime2}-\left[  D_{t}D_{s}\tau
_{n}^{\prime}\cdot\tau_{n}^{\prime}-2\tau_{n+1}^{\prime}\tau_{n-1}^{\prime
}+2\tau_{n}^{\prime2}\right]  \tau_{n}^{2}.
\]
Then we have the identity(Page 179, \cite{Hirota})
\begin{align*}
\frac{1}{2}\mathcal{P}  &  =D_{t}\left[  D_{s}\tau_{n}\cdot\tau_{n}^{\prime
}-\lambda\tau_{n+1}\cdot\tau_{n-1}^{\prime}+\lambda\tau_{n}\tau_{n}^{\prime
}\right]  \cdot\left(  \tau_{n}^{\prime}\tau_{n}\right) \\
&  +\lambda\left[  D_{t}\tau_{n+1}\cdot\tau_{n}^{\prime}+\lambda^{-1}\tau
_{n}\tau_{n+1}^{\prime}-\lambda^{-1}\tau_{n+1}\tau_{n}^{\prime}\right]
\tau_{n-1}^{\prime}\tau_{n}\\
&  -\lambda\left[  D_{t}\tau_{n}\cdot\tau_{n-1}^{\prime}+\lambda^{-1}%
\tau_{n-1}\tau_{n}^{\prime}-\lambda^{-1}\tau_{n}\tau_{n-1}^{\prime}\right]
\tau_{n}^{\prime}\tau_{n+1}.
\end{align*}
Here $\lambda$ is a free parameter. From this, we get the following B\"acklund
transformation between $\left\{  \tau_{n}\right\}  $ and $\left\{  \tau
_{n}^{\prime}\right\}  :$%
\begin{equation}
\left\{
\begin{array}
[c]{l}%
D_{s}\tau_{n}\cdot\tau_{n}^{\prime}-\lambda\tau_{n+1}\cdot\tau_{n-1}^{\prime
}+\lambda\tau_{n}\tau_{n}^{\prime}=0,\\
D_{t}\tau_{n+1}\cdot\tau_{n}^{\prime}+\lambda^{-1}\tau_{n}\tau_{n+1}^{\prime
}-\lambda^{-1}\tau_{n+1}\tau_{n}^{\prime}=0.
\end{array}
\right.  \label{b}%
\end{equation}
If $\left\{  \tau_{n}\right\}  $ is a solution of $\left(  \ref{so}\right)  $
and $\left\{  \tau_{n}\right\}  ,\left\{  \tau_{n}^{\prime}\right\}  $ satisfy
$\left(  \ref{b}\right)  ,$ then $\left\{  \tau_{n}^{\prime}\right\}  $ is
also a solution of $\left(  \ref{so}\right)  .$

In the rest of the paper, we choose $\lambda=\sqrt{2}+1.$ The B\"acklund
transformation from $\kappa_{n}$ to $\omega_{n}$ is given by
\begin{equation}
\left\{
\begin{array}
[c]{l}%
D_{s}\kappa_{n}\cdot\omega_{n}=\lambda\left(  \kappa_{n+1}\omega_{n-1}%
-\kappa_{n}\omega_{n}\right)  ,\\
D_{t}\kappa_{n+1}\cdot\omega_{n}=-\lambda^{-1}\left(  \kappa_{n}\omega
_{n+1}-\kappa_{n+1}\omega_{n}\right)  .
\end{array}
\right.  \label{b1}%
\end{equation}
The B\"acklund transformation from $\omega_{n}$ to $\theta_{n}$ is
\begin{equation}
\left\{
\begin{array}
[c]{l}%
D_{s}\omega_{n}\cdot\theta_{n}=\lambda^{-1}\left(  \omega_{n+1}\theta
_{n-1}-\omega_{n}\theta_{n}\right)  ,\\
D_{t}\omega_{n+1}\cdot\theta_{n}=-\lambda\left(  \omega_{n}\theta_{n+1}%
-\omega_{n+1}\theta_{n}\right)  .
\end{array}
\right.  \label{b2}%
\end{equation}

We refer to \cite{Rogers} for related results on the B\"acklund transformation
of $1+1$ Toda lattice and other integrable systems.

\section{The linearized B\"acklund transformation between $\omega$ and
$\theta$}

In this section, we study the linearized B\"acklund transformation between
$\omega$ and $\theta.$ The linearization of the system $\left(  \ref{b2}%
\right)  $ is
\begin{equation}
\left\{
\begin{array}
[c]{l}%
\partial_{s}\phi_{n}\theta_{n}-\phi_{n}\partial_{s}\theta_{n}-\lambda
^{-1}\left(  \phi_{n+1}\theta_{n-1}-\phi_{n}\theta_{n}\right) \\
=-\partial_{s}\omega_{n}\eta_{n}+\omega_{n}\partial_{s}\eta_{n}+\lambda
^{-1}\left(  \omega_{n+1}\eta_{n-1}-\omega_{n}\eta_{n}\right)  ,\\
\partial_{t}\phi_{n}\theta_{n-1}-\phi_{n}\partial_{t}\theta_{n-1}%
+\lambda\left(  \phi_{n-1}\theta_{n}-\phi_{n}\theta_{n-1}\right) \\
=-\partial_{t}\omega_{n}\eta_{n-1}+\omega_{n}\partial_{t}\eta_{n-1}%
-\lambda\left(  \omega_{n-1}\eta_{n}-\omega_{n}\eta_{n-1}\right)  .
\end{array}
\right.  \label{l2}%
\end{equation}
Dividing the first equation by $\theta_{n}$ and the second one by
$\theta_{n-1},$ $\left(  \ref{l2}\right)  $ can be rewritten as
\begin{equation}
\left\{
\begin{array}
[c]{c}%
\left(  F_{1}\phi\right)  _{n}=\left(  G_{1}\eta\right)  _{n},\\
\left(  M_{1}\phi\right)  _{n}=\left(  N_{1}\eta\right)  _{n},
\end{array}
\right.  \label{s2}%
\end{equation}
where
\begin{align*}
\left(  F_{1}\phi\right)  _{n}  &  =\partial_{x}\phi_{n}-\left(
\frac{\partial_{s}\theta_{n}}{\theta_{n}}+\frac{\partial_{t}\theta_{n-1}%
}{\theta_{n-1}}+2\right)  \phi_{n}-\lambda^{-1}\frac{\theta_{n-1}}{\theta_{n}%
}\phi_{n+1}+\lambda\frac{\theta_{n}}{\theta_{n-1}}\phi_{n-1},\\
\left(  M_{1}\phi\right)  _{n}  &  =\frac{1}{i}\partial_{y}\phi_{n}-\left(
\frac{\partial_{s}\theta_{n}}{\theta_{n}}-\frac{\partial_{t}\theta_{n-1}%
}{\theta_{n-1}}-2\sqrt{2}\right)  \phi_{n}-\lambda^{-1}\phi_{n+1}\frac
{\theta_{n-1}}{\theta_{n}}-\lambda\phi_{n-1}\frac{\theta_{n}}{\theta_{n-1}},
\end{align*}
and%
\begin{align*}
\left(  G_{1}\eta\right)  _{n}  &  =\frac{\omega_{n}}{\theta_{n}}\partial
_{s}\eta_{n}+\frac{\omega_{n}}{\theta_{n-1}}\partial_{t}\eta_{n-1}+\left(
-\frac{\partial_{s}\omega_{n}}{\theta_{n}}-\lambda^{-1}\frac{\omega_{n}%
}{\theta_{n}}-\lambda\frac{\omega_{n-1}}{\theta_{n-1}}\right)  \eta_{n}\\
&  +\left(  \lambda\frac{\omega_{n+1}}{\theta_{n}}-\frac{\partial_{t}%
\omega_{n}}{\theta_{n-1}}+\frac{\lambda\omega_{n}}{\theta_{n-1}}\right)
\eta_{n-1},
\end{align*}%
\begin{align*}
\left(  N_{1}\eta\right)  _{n}  &  =\frac{\omega_{n}}{\theta_{n}}\partial
_{s}\eta_{n}-\frac{\omega_{n}}{\theta_{n-1}}\partial_{t}\eta_{n-1}+\left(
-\frac{\partial_{s}\omega_{n}}{\theta_{n}}-\lambda^{-1}\frac{\omega_{n}%
}{\theta_{n}}+\lambda\frac{\omega_{n-1}}{\theta_{n-1}}\right)  \eta_{n}\\
&  +\left(  \lambda\frac{\omega_{n+1}}{\theta_{n}}+\frac{\partial_{t}%
\omega_{n}}{\theta_{n-1}}-\frac{\lambda\omega_{n}}{\theta_{n-1}}\right)
\eta_{n-1}.
\end{align*}
In this section, we would like to prove the following

\begin{proposition}
\label{P2}Let $\left\{  \eta_{n}\right\}  $ be given by Lemma \ref{linear}.
Then $\left(  \ref{s2}\right)  $ has a solution $\left\{  \phi_{n}\right\}  $
with $\phi_{n+1}\left(  x,y\right)  =\phi_{n}\left(  x+\frac{1}{2\sqrt{2}%
},y\right)  $ and
\begin{equation}
\left\vert \phi_{0}\left(  x,y\right)  \right\vert \leq C\left(  1+x^{2}%
+y^{2}\right)  ^{\frac{5}{8}}. \label{f0}%
\end{equation}

\end{proposition}

We remark that the exponent $\frac{5}{8}$ is not optimal, but it is suffice
for our use in the proof of Theorem \ref{Main}.

To prove Proposition \ref{P2}, we will use the Fourier transform. Let us use
$\phi^{\symbol{94}}$ to denote the Fourier transform of a generalized function
$\phi=\phi\left(  x,y\right)  $ with respect to the $x$ variable. In
particular, if $\phi$ is a regular function, then
\[
\phi^{\symbol{94}}\left(  \xi,y\right)  =\int_{\mathbb{R}}e^{-2\pi ix\xi}%
\phi\left(  x,y\right)  dx.
\]
The Heaviside step function will be denoted by $u.$ That is,
\[
u\left(  x\right)  :=\left\{
\begin{array}
[c]{c}%
0,x<0,\\
1,x\geq0.
\end{array}
\right.
\]
We will frequently use the following formulas for the Fourier transform(See
Appendix 2 of the book \cite{F1}).

\begin{lemma}
\label{f1}Let $a_{1}\in\mathbb{R},a_{2}>0.$ Then
\[
\left(  \frac{1}{x+a_{1}-a_{2}i}\right)  ^{\symbol{94}}=2\pi ie^{2\pi i\left(
a_{1}-a_{2}i\right)  \xi}u\left(  -\xi\right)  ,
\]
and
\[
\left(  \frac{1}{x+a_{1}+a_{2}i}\right)  ^{\symbol{94}}=\left(  -2\pi
i\right)  e^{2\pi i\left(  a_{1}+a_{2}i\right)  \xi}u\left(  \xi\right)  .
\]

\end{lemma}

\begin{lemma}
\label{f2}Let $a_{1}\in\mathbb{R}.$ Then
\[
\left(  \frac{1}{x+a_{1}}\right)  ^{\symbol{94}}=-\pi ie^{2\pi ia_{1}\xi
}\text{Sgn}\xi.
\]

\end{lemma}

\begin{remark}
From Lemma \ref{f1} and Lemma \ref{f2}, we see that formally,
\[
\left(  \frac{1}{x}\right)  ^{\symbol{94}}=\lim_{\varepsilon\rightarrow0}%
\frac{1}{2}\left(  \frac{1}{x+\varepsilon i}+\frac{1}{x-\varepsilon i}\right)
.
\]
Here the distribution $\frac{1}{x}$ is defined to be the derivative of
$\ln\left\vert x\right\vert .$
\end{remark}

\begin{lemma}
\label{L}Suppose $a_{1}\in\mathbb{R},a_{2}>0,$ and $a_{3}\in\mathbb{C}$. Then%
\begin{align*}
\left(  \frac{\left(  x+a_{1}\right)  +a_{3}}{\left(  x+a_{1}\right)
^{2}+a_{2}^{2}}\right)  ^{\symbol{94}}  &  =\left(  \frac{1}{2}-\frac{a_{3}%
}{2a_{2}i}\right)  \left(  -2\pi i\right)  e^{2\pi i\left(  a_{1}%
+a_{2}i\right)  \xi}u\left(  \xi\right) \\
&  +\left(  \frac{1}{2}+\frac{a_{3}}{2a_{2}i}\right)  \left(  2\pi i\right)
e^{2\pi i\left(  a_{1}-a_{2}i\right)  \xi}u\left(  -\xi\right)  .
\end{align*}

\end{lemma}

\begin{proof}
We have
\begin{align*}
\frac{\left(  x+a_{1}\right)  +a_{3}}{\left(  x+a_{1}\right)  ^{2}+a_{2}^{2}}
&  =\frac{1}{2}\frac{1}{x+a_{1}+a_{2}i}+\frac{1}{2}\frac{1}{x+a_{1}-a_{2}i}\\
&  +\frac{a_{3}}{2a_{2}i}\left(  \frac{1}{x+a_{1}-a_{2}i}-\frac{1}%
{x+a_{1}+a_{2}i}\right) \\
&  =\left(  \frac{1}{2}-\frac{a_{3}}{2a_{2}i}\right)  \frac{1}{x+a_{1}+a_{2}%
i}\\
&  +\left(  \frac{1}{2}+\frac{a_{3}}{2a_{2}i}\right)  \frac{1}{x+a_{1}-a_{2}%
i}.
\end{align*}
The result then follows from Lemma \ref{f1}.
\end{proof}

We define
\[
\alpha=\frac{n}{2\sqrt{2}},\beta=\sqrt{\frac{y^{2}}{2}+\frac{1}{32}}%
,b=-\frac{y}{2},
\]
and
\[
c=\frac{1}{2\sqrt{2}},\alpha_{1}=\alpha-c.
\]
Set
\begin{align*}
A_{1}  &  =\frac{1}{2}-\frac{b}{2\beta},A_{2}=-\frac{\lambda}{\sqrt{2}}\left(
\frac{1}{2}-\frac{\sqrt{2}}{16\beta i}\right)  ,\\
A_{3}  &  =-\left(  \frac{1}{2}+\frac{b}{2\beta}\right)  ,A_{4}=\frac{\lambda
}{\sqrt{2}}\left(  \frac{1}{2}+\frac{\sqrt{2}}{16\beta i}\right)  .
\end{align*}

\begin{lemma}
\label{L1}Let $\theta_{n}$ be defined by $\left(  \ref{th}\right)  .$ Then
\[
\left[  \frac{\partial_{s}\theta_{n}}{\theta_{n}}\right]  ^{\symbol{94}}=-2\pi
iA_{1}e^{2\pi i\left(  \alpha+\beta i\right)  \xi}u\left(  \xi\right)  -2\pi
iA_{3}e^{2\pi i\left(  \alpha-\beta i\right)  \xi}u\left(  -\xi\right)  ,
\]%
\[
\left[  \frac{\partial_{t}\theta_{n-1}}{\theta_{n-1}}\right]  ^{\symbol{94}%
}=2\pi iA_{3}e^{2\pi i\left(  \alpha_{1}+\beta i\right)  \xi}u\left(
\xi\right)  +2\pi iA_{1}e^{2\pi i\left(  \alpha_{1}-\beta i\right)  \xi
}u\left(  -\xi\right)  ,
\]
and%
\[
\left[  \frac{\theta_{n-1}}{\theta_{n}}\right]  ^{\symbol{94}}=\delta+2\pi
i\frac{A_{4}}{\lambda}e^{2\pi i\left(  \alpha+\beta i\right)  \xi}u\left(
\xi\right)  +2\pi i\frac{A_{2}}{\lambda}e^{2\pi i\left(  \alpha-\beta
i\right)  \xi}u\left(  -\xi\right)  ,
\]
and%
\[
\left[  \frac{\theta_{n}}{\theta_{n-1}}\right]  ^{\symbol{94}}=\delta+2\pi
i\frac{A_{2}}{\lambda}e^{2\pi i\left(  \alpha_{1}+\beta i\right)  \xi}u\left(
\xi\right)  +2\pi i\frac{A_{4}}{\lambda}e^{2\pi i\left(  \alpha_{1}-\beta
i\right)  \xi}u\left(  -\xi\right)  .
\]

\end{lemma}

\begin{proof}
We compute%
\begin{align*}
\frac{\partial_{s}\theta_{n}}{\theta_{n}}  &  =\frac{2\sqrt{2}\left(
2\sqrt{2}x+n\right)  -4yi}{\left(  2\sqrt{2}x+n\right)  ^{2}+4y^{2}+\frac
{1}{4}}\\
&  =\frac{x+\frac{n}{2\sqrt{2}}-\frac{yi}{2}}{\left(  x+\frac{n}{2\sqrt{2}%
}\right)  ^{2}+\frac{y^{2}}{2}+\frac{1}{32}}.
\end{align*}
Similarly,%
\begin{align*}
\frac{\partial_{t}\theta_{n}}{\theta_{n}}  &  =\frac{2\sqrt{2}\left(
2\sqrt{2}x+n\right)  +4yi}{\left(  2\sqrt{2}x+n\right)  ^{2}+4y^{2}+\frac
{1}{4}}\\
&  =\frac{x+\frac{n}{2\sqrt{2}}+\frac{yi}{2}}{\left(  x+\frac{n}{2\sqrt{2}%
}\right)  ^{2}+2y^{2}+\frac{1}{32}}.
\end{align*}
Moreover,
\begin{align*}
\frac{\theta_{n-1}}{\theta_{n}}  &  =\frac{\theta_{n}-\left[  2\left(
2\sqrt{2}x+n\right)  -1\right]  }{\theta_{n}}=1-\frac{2\left(  2\sqrt
{2}x+n\right)  -1}{\left(  2\sqrt{2}x+n\right)  ^{2}+4y^{2}+\frac{1}{4}}\\
&  =1-\frac{1}{\sqrt{2}}\frac{\left(  x+\frac{n}{2\sqrt{2}}\right)
-\frac{\sqrt{2}}{8}}{\left(  x+\frac{n}{2\sqrt{2}}\right)  ^{2}+\frac{y^{2}%
}{2}+\frac{1}{32}}.
\end{align*}
We also have%
\begin{align*}
\frac{\theta_{n}}{\theta_{n-1}}  &  =1+\frac{2\left(  2\sqrt{2}x+n\right)
-1}{\tau_{n-1}^{\prime}}\\
&  =1+\frac{1}{\sqrt{2}}\frac{\left(  x+\frac{n-1}{2\sqrt{2}}\right)
+\frac{\sqrt{2}}{8}}{\left(  x+\frac{n-1}{2\sqrt{2}}\right)  ^{2}+\frac{y^{2}%
}{2}+\frac{1}{32}}.
\end{align*}
The desired results then follow from direct application of Lemma \ref{L}.
\end{proof}

To proceed, we need to introduce some notations. We define
\[
\gamma^{\ast}:=\frac{A_{2}}{A_{3}\lambda^{2}}\ =\frac{A_{1}}{A_{4}},
\]
and
\[
\gamma:=\frac{A_{4}}{A_{1}\lambda^{2}}=\frac{A_{3}}{A_{2}}.
\]
Then $\gamma\gamma^{\ast}=\lambda^{-2}.$ Note that $\left\vert \gamma
\right\vert <1,$ and $\left\vert \gamma\right\vert \rightarrow1$ as
$y\rightarrow+\infty.$

We introduce the functions
\[
2\pi iP\left(  \xi\right)  =2\pi i\xi-2-\lambda^{-1}e^{\frac{\pi i\xi}%
{\sqrt{2}}}+\lambda e^{-\frac{\pi i\xi}{\sqrt{2}}},
\]%
\[
Q\left(  \xi\right)  =\left(  1-\gamma e^{\frac{\pi i\xi}{\sqrt{2}}}\right)
\left(  A_{1}+A_{2}e^{-\frac{\pi i\xi}{\sqrt{2}}}\right)  -\left(
1-\gamma^{\ast}e^{\frac{\pi i\xi}{\sqrt{2}}}\right)  \left(  A_{3}%
+A_{4}e^{-\frac{\pi i\xi}{\sqrt{2}}}\right)  ,
\]%
and
\[
J\left(  \xi\right)  =\left(  \frac{1-\gamma^{\ast}e^{\frac{\pi i\xi}{\sqrt
{2}}}}{1-\gamma e^{\frac{\pi i\xi}{\sqrt{2}}}}e^{-4\pi\beta\xi}\right)
^{\prime},
\]%
\[
R\left(  \xi\right)  =\left(  1-\gamma e^{\frac{\pi i\xi}{\sqrt{2}}}\right)
\left(  A_{3}+A_{4}e^{-\frac{\pi i\xi}{\sqrt{2}}}\right)  e^{4\pi\beta\xi}.
\]
We then define
\begin{equation}
h\left(  \xi\right)  =\int_{-\infty}^{\xi}\ \left(  1-\gamma e^{\frac{\pi
is}{\sqrt{2}}}\right)  e^{2\pi i\left(  \alpha+\beta i\right)  \left(
-s\right)  }\phi_{n}^{\symbol{94}}\left(  s\right)  ds. \label{h}%
\end{equation}
Let
\begin{equation}
g\left(  \xi\right)  =\int_{\xi}^{+\infty}h\left(  s\right)  J\left(
s\right)  ds. \label{ggg}%
\end{equation}

\begin{proposition}
\label{F1}%
\begin{equation}
\left[  \left(  F_{1}\phi\right)  _{n}\right]  ^{\symbol{94}}=-\frac{2\pi
ie^{2\pi i\left(  \alpha+\beta i\right)  \xi}}{J}\left(  Pg^{\prime\prime
}+\left(  Q-P\frac{J^{\prime}}{J}\right)  g^{\prime}+RJg\right)  . \label{gg}%
\end{equation}

\end{proposition}

\begin{proof}
Using Lemma \ref{L1}, we find that the Fourier transform of $\left(  F_{1}%
\phi\right)  _{n}$ is equal to
\begin{align*}
&  2\pi i\xi\phi_{n}^{\symbol{94}}-2\phi_{n}^{\symbol{94}}-\lambda^{-1}%
\phi_{n+1}^{\symbol{94}}+\lambda\phi_{n-1}^{\symbol{94}}\\
&  -\phi_{n}^{\symbol{94}}\ast\left[  -2\pi iA_{1}e^{2\pi i\left(
\alpha+\beta i\right)  \xi}u\left(  \xi\right)  -2\pi iA_{3}e^{2\pi i\left(
\alpha-\beta i\right)  \xi}u\left(  -\xi\right)  \right] \\
&  -\phi_{n}^{\symbol{94}}\ast\left[  2\pi iA_{3}e^{2\pi i\left(  \alpha
_{1}+\beta i\right)  \xi}u\left(  \xi\right)  +2\pi iA_{1}e^{2\pi i\left(
\alpha_{1}-\beta i\right)  \xi}u\left(  -\xi\right)  \right] \\
&  -\lambda^{-1}\phi_{n+1}^{\symbol{94}}\ast\left[  2\pi i\frac{A_{4}}%
{\lambda}e^{2\pi i\left(  \alpha+\beta i\right)  \xi}u\left(  \xi\right)
+2\pi i\frac{A_{2}}{\lambda}e^{2\pi i\left(  \alpha-\beta i\right)  \xi
}u\left(  -\xi\right)  \right] \\
&  +\lambda\phi_{n-1}^{\symbol{94}}\ast\left[  2\pi i\frac{A_{2}}{\lambda
}e^{2\pi i\left(  \alpha_{1}+\beta i\right)  \xi}u\left(  \xi\right)  +2\pi
i\frac{A_{4}}{\lambda}e^{2\pi i\left(  \alpha_{1}-\beta i\right)  \xi}u\left(
-\xi\right)  \right]  .
\end{align*}
We calculate
\begin{align*}
&  \phi_{n-1}^{\symbol{94}}\ast\left[  e^{2\pi i\left(  \alpha_{1}+\beta
i\right)  \xi}u\left(  \xi\right)  \right] \\
&  =\int_{-\infty}^{\xi}e^{2\pi i\left(  \alpha_{1}+\beta i\right)  \left(
\xi-s\right)  }e^{-2\pi ics}\phi_{n}^{\symbol{94}}\left(  s\right)  ds.
\end{align*}
Using the fact that $\alpha_{1}=\alpha-c,$ we obtain%
\[
\phi_{n-1}^{\symbol{94}}\ast\left[  e^{2\pi i\left(  \alpha_{1}+\beta
i\right)  \xi}u\left(  \xi\right)  \right]  =e^{2\pi i\left(  \alpha_{1}+\beta
i\right)  \xi}\int_{-\infty}^{\xi}e^{2\pi i\left(  \alpha+\beta i\right)
\left(  -s\right)  }\phi_{n}^{\symbol{94}}\left(  s\right)  ds.
\]
Similarly,
\begin{align*}
\phi_{n+1}^{\symbol{94}}\ast\left[  e^{2\pi i\left(  \alpha+\beta i\right)
\xi}u\left(  \xi\right)  \right]   &  =\int_{-\infty}^{\xi}e^{2\pi i\left(
\alpha+\beta i\right)  \left(  \xi-s\right)  }e^{2\pi ics}\phi_{n}%
^{\symbol{94}}\left(  s\right)  ds\\
&  =e^{2\pi i\left(  \alpha+\beta i\right)  \xi}\int_{-\infty}^{\xi}e^{2\pi
i\left(  \alpha_{1}+\beta i\right)  \left(  -s\right)  }\phi_{n}^{\symbol{94}%
}\left(  s\right)  ds.
\end{align*}
It follows that $\frac{1}{2\pi i}\left[  \left(  F_{1}\phi\right)
_{n}\right]  ^{\symbol{94}}$ is equal to
\begin{align}
&  P\left(  \xi\right)  \phi_{n}^{\symbol{94}}+\left(  A_{1}+A_{2}e^{-2\pi
ic\xi}\right)  e^{2\pi i\left(  \alpha+\beta i\right)  \xi}h\nonumber\\
&  +\left(  A_{3}+A_{4}e^{-2\pi ic\xi}\right)  e^{2\pi i\left(  \alpha-\beta
i\right)  \xi}\int_{\xi}^{+\infty}\ \left[  1-\gamma^{\ast}e^{2\pi
ics}\right]  e^{2\pi i\left(  \alpha-\beta i\right)  \left(  -s\right)  }%
\phi_{n}^{\symbol{94}}\left(  s\right)  ds\nonumber\\
&  =P\left(  \xi\right)  \frac{e^{2\pi i\left(  \alpha+\beta i\right)  \xi}%
}{1-\gamma e^{2\pi ic\xi}}h^{\prime}+\left(  A_{1}+A_{2}e^{-2\pi ic\xi
}\right)  e^{2\pi i\left(  \alpha+\beta i\right)  \xi}h\label{eq2}\\
&  +\left(  A_{3}+A_{4}e^{-2\pi ic\xi}\right)  e^{2\pi i\left(  \alpha-\beta
i\right)  \xi}\int_{\xi}^{+\infty}\ \left[  1-\gamma^{\ast}e^{2\pi
ics}\right]  e^{2\pi i\left(  \alpha-\beta i\right)  \left(  -s\right)  }%
\phi_{n}^{\symbol{94}}\left(  s\right)  ds.\nonumber
\end{align}
The last term is equal to
\begin{align*}
&  \left(  A_{3}+A_{4}e^{-2\pi ic\xi}\right)  e^{2\pi i\left(  \alpha-\beta
i\right)  \xi}\int_{\xi}^{+\infty}\ \left[  1-\gamma^{\ast}e^{2\pi
ics}\right]  e^{2\pi i\left(  \alpha-\beta i\right)  \left(  -s\right)  }%
\phi_{n}^{\symbol{94}}\left(  s\right)  ds\\
&  =\left(  A_{3}+A_{4}e^{-2\pi ic\xi}\right)  e^{2\pi i\left(  \alpha-\beta
i\right)  \xi}\int_{\xi}^{+\infty}\ \left[  1-\gamma^{\ast}e^{2\pi
ics}\right]  e^{2\pi i\left(  \alpha-\beta i\right)  \left(  -s\right)  }%
\frac{e^{2\pi i\left(  \alpha+\beta i\right)  s}}{1-\gamma e^{2\pi ics}%
}h^{\prime}ds\\
&  =\left(  A_{3}+A_{4}e^{-2\pi ic\xi}\right)  e^{2\pi i\left(  \alpha-\beta
i\right)  \xi}\int_{\xi}^{+\infty}\ \frac{1-\gamma^{\ast}e^{2\pi ics}%
}{1-\gamma e^{2\pi ics}}e^{4\pi i\left(  \beta i\right)  s}h^{\prime}ds\\
&  =-\frac{1-\gamma^{\ast}e^{2\pi ic\xi}}{1-\gamma e^{2\pi ic\xi}}\left(
A_{3}+A_{4}e^{-2\pi ic\xi}\right)  e^{2\pi i\left(  \alpha+\beta i\right)
\xi}h\left(  \xi\right) \\
&  -\left(  A_{3}+A_{4}e^{-2\pi ic\xi}\right)  e^{2\pi i\left(  \alpha-\beta
i\right)  \xi}\int_{\xi}^{+\infty}h\left(  s\right)  J\left(  s\right)  ds,
\end{align*}
Insert this identity into $\left(  \ref{eq2}\right)  ,$ we find that
\begin{align*}
&  \frac{1}{2\pi ie^{2\pi i\left(  \alpha+\beta i\right)  \xi}}\left[  \left(
F_{1}\phi\right)  _{n}\right]  ^{\symbol{94}}=\\
&  \frac{P\left(  \xi\right)  h^{\prime}}{1-\gamma e^{2\pi ic\xi}}+\left(
A_{1}+A_{2}e^{-2\pi ic\xi}\right)  h-\left(  A_{3}+A_{4}e^{-2\pi ic\xi
}\right)  \frac{1-\gamma^{\ast}e^{2\pi ic\xi}}{1-\gamma e^{2\pi ic\xi}}h\\
&  -\left(  A_{3}+A_{4}e^{-2\pi ic\xi}\right)  e^{4\pi\beta\xi}\int_{\xi
}^{+\infty}h\left(  s\right)  J\left(  \xi\right)  ds.
\end{align*}
Inserting the relation $g^{\prime}\left(  \xi\right)  =-h\left(  \xi\right)
J\left(  \xi\right)  $ into this equation, we get the identity $\left(
\ref{gg}\right)  .$
\end{proof}

Since $\frac{1}{P}$ has a singularity around $0,$ it will be important to
understand the behavior of $Q$ and $RJ$ around $0.$

\begin{lemma}
For all $y\in\mathbb{R},$ we have%
\[
Q\left(  0\right)  =0\text{ and }Q^{\prime}\left(  0\right)  =\pi i.
\]

\end{lemma}

\begin{proof}
Direct computation using the definition of $\gamma$ and $\gamma^{\ast}$ shows
that
\begin{align*}
Q\left(  0\right)   &  =\left(  A_{1}+A_{2}\right)  \left(  1-\gamma\right)
-\left(  A_{3}+A_{4}\right)  \left(  1-\gamma^{\ast}\right) \\
&  =A_{1}+A_{2}-\frac{1}{\lambda^{2}}A_{4}-A_{3}-A_{3}-A_{4}+\frac{1}%
{\lambda^{2}}A_{2}+A_{1}\\
&  =0.
\end{align*}
On the othe hand, since $A_{i}$ does not depend on $\xi,$ we calucate
\begin{align*}
Q^{\prime}\left(  \xi\right)   &  =A_{2}\left(  -2\pi ic\right)  e^{-2\pi
ic\xi}-A_{1}\gamma\left(  2\pi ic\right)  e^{2\pi ic\xi}\\
&  +A_{3}\gamma^{\ast}\left(  2\pi ic\right)  e^{2\pi ic\xi}-A_{4}\left(
-2\pi ic\right)  e^{-2\pi ic\xi}.
\end{align*}
Therefore,
\begin{align*}
\frac{Q^{\prime}\left(  0\right)  }{2\pi ic}  &  =-A_{2}-A_{1}\gamma
+A_{3}\gamma^{\ast}+A_{4}\\
&  =\lambda\frac{1}{\sqrt{2}}\left(  \frac{1}{2}-\frac{\sqrt{2}}{16\beta
i}\right)  -\lambda^{-1}\frac{1}{\sqrt{2}}\left(  \frac{1}{2}+\frac{\sqrt{2}%
}{16\beta i}\right) \\
&  -\lambda^{-1}\frac{1}{\sqrt{2}}\left(  \frac{1}{2}-\frac{\sqrt{2}}{16\beta
i}\right)  +\lambda\frac{1}{\sqrt{2}}\left(  \frac{1}{2}+\frac{\sqrt{2}%
}{16\beta i}\right) \\
&  =\sqrt{2}.
\end{align*}
The proof is completed.
\end{proof}

\begin{lemma}
\label{J}For all $y\in\mathbb{R},$ we have%
\[
J\left(  0\right)  =0.
\]

\end{lemma}

\begin{proof}
Since%
\[
\left(  \frac{1-\gamma^{\ast}e^{2\pi ics}}{1-\gamma e^{2\pi ics}}\right)
^{\prime}=\frac{\gamma-\gamma^{\ast}}{\left(  1-\gamma e^{2\pi ics}\right)
^{2}}2\pi ice^{2\pi ics},
\]
we have
\begin{align*}
&  J\left(  s\right)  =\frac{\gamma-\gamma^{\ast}}{\left(  1-\gamma e^{2\pi
ics}\right)  ^{2}}2\pi ice^{2\pi ics}e^{-4\pi\beta s}+\frac{1-\gamma^{\ast
}e^{2\pi ics}}{1-\gamma e^{2\pi ics}}\left(  -4\pi\beta\right)  e^{-4\pi\beta
s}\\
&  =\frac{2\pi e^{-4\pi\beta s}}{\left(  1-\gamma e^{2\pi ics}\right)  ^{2}%
}\left[  \left(  \gamma-\gamma^{\ast}\right)  ice^{2\pi ics}+\left(  1-\gamma
e^{2\pi ics}\right)  \left(  1-\gamma^{\ast}e^{2\pi ics}\right)  \left(
-2\beta\right)  \right]  .
\end{align*}
Letting $s=0$, it follows that
\[
\frac{\left(  1-\gamma\right)  ^{2}}{2\pi}J\left(  0\right)  =\frac{\left(
\gamma-\gamma^{\ast}\right)  i}{2\sqrt{2}}-\left(  1-\gamma\right)  \left(
1-\gamma^{\ast}\right)  2\beta.
\]
Using software such as $\mathit{Mathematica,}$ one can directly verify that
$J\left(  0\right)  =0.$
\end{proof}

Consider the equation%
\[
Pg^{\prime\prime}+\left(  Q-P\frac{J^{\prime}}{J}\right)  g^{\prime}+RJg=0.
\]
Let us write it as
\begin{equation}
P_{1}g^{\prime\prime}+Q_{1}g^{\prime}+R_{1}g=0. \label{g}%
\end{equation}
Here $P_{1}=P,$ $Q_{1}=Q-P\frac{J^{\prime}}{J},$ and $R_{1}=RJ.$ Note that
under the transformation $\phi_{n}\rightarrow h\rightarrow g,$ the function
$\phi_{n}=\omega_{n}$ corresponding to $g=0.$

We are interested in the asymptotic behavior of the solutions to this
equation$.$ At this point, it is worth pointing out that according to Lemma
\ref{J}, the function $Q_{1}$ has singularities at $\xi=2\sqrt{2}\pi
j,j\in\mathbb{N}.$

\begin{lemma}
The equation $\left(  \ref{g}\right)  $ has two solutions $g_{1},g_{2},$
satisfying
\[
g_{1}\left(  \xi\right)  =1+O\left(  \xi\right)  ,\text{ as }\xi\rightarrow0,
\]
and
\[
g_{2}\left(  \xi\right)  =\xi^{-2}+O\left(  \xi^{-1}\right)  ,\text{ as }%
\xi\rightarrow0.
\]
Moreover, $g_{1},g_{2}$ are smooth at $\xi\neq0.$
\end{lemma}

\begin{proof}
For $\xi$ close to $0,$
\[
P\left(  \xi\right)  =\frac{\pi i}{4}\xi^{2}+O\left(  \xi^{3}\right)  ,\text{
}Q\left(  \xi\right)  =\pi i\xi+O\left(  \xi^{2}\right)  .
\]
Hence%
\[
Q_{1}\left(  \xi\right)  =\frac{3\pi i}{4}\xi+O\left(  \xi^{2}\right)  .
\]
The existence of $g_{1},g_{2}$ then follows from a perturbation argument.

Near the points $\xi_{j}=2\sqrt{2}\pi j,$ $j\in\mathbb{N}\backslash\left\{
0\right\}  ,$ the equation $\left(  \ref{g}\right)  $ can be written as
\[
g^{\prime\prime}+\left(  \left(  \xi-\xi_{j}\right)  ^{-1}+O\left(  1\right)
\right)  g^{\prime}+O\left(  \xi-\xi_{j}\right)  g=0.
\]
We then can deduce the existence of two smooth linearly independent solutions
using perturbation arguments again. Indeed, we can also find the asymptotic
behavior of these solutions.
\end{proof}

Taking Fourier transform for the equation $\left(  F_{1}\phi\right)
_{n}=\left(  G_{1}\eta\right)  _{n},$ using Proposition \ref{F1}, we obtain
\begin{equation}
P_{1}g^{\prime\prime}+Q_{1}g^{\prime}+R_{1}g=-\frac{J}{2\pi ie^{2\pi i\left(
\alpha+\beta i\right)  \xi}}\left[  \left(  G_{1}\eta\right)  _{n}\right]
^{\symbol{94}}:=B\left(  \xi,y\right)  . \label{2}%
\end{equation}
Variation of parameter formula tells us that equation $\left(  \ref{2}\right)
$ has a solution of the form
\[
g^{\ast}\left(  \xi,y\right)  =g_{2}\left(  \xi,y\right)  \int_{+\infty}^{\xi
}\frac{g_{1}\left(  s,y\right)  }{W\left(  s,y\right)  }\frac{B\left(
s,y\right)  }{P_{1}\left(  s,y\right)  }ds-g_{1}\left(  \xi,y\right)
\int_{+\infty}^{\xi}\frac{g_{2}\left(  s,y\right)  }{W\left(  s,y\right)
}\frac{B\left(  s,y\right)  }{P_{1}\left(  s,y\right)  }ds.
\]
Here $W\left(  s,y\right)  $ is the Wronskian of $g_{1}$ and $g_{2}.$ Let%
\[
h^{\ast}\left(  \xi,y\right)  =-\frac{g^{\prime}\left(  \xi,y\right)
}{J\left(  \xi,y\right)  },
\]
and we define $\phi_{n}^{\ast}$ by
\[
\left(  \phi_{n}^{\ast}\right)  ^{\symbol{94}}=\frac{h^{\ast\prime}}{\left(
1-\gamma e^{2\pi ic\xi}\right)  e^{2\pi i\left(  \alpha+\beta i\right)
\left(  -\xi\right)  }}.
\]

With these preparation, now we seek a solution $\phi=\left\{  \phi
_{n}\right\}  $ for the system%
\[
\left\{
\begin{array}
[c]{c}%
\left(  F_{1}\phi\right)  _{n}=\left(  G_{1}\eta\right)  _{n},\\
\left(  M_{1}\phi\right)  _{n}=\left(  N_{1}\eta\right)  _{n}.
\end{array}
\right.
\]
in the form
\begin{equation}
\phi_{n}=\phi_{n}^{\ast}+\zeta\left(  y\right)  \omega_{n}. \label{fi1}%
\end{equation}

\begin{lemma}
There exists a solution $\zeta$ to the equation
\[
\left(  M_{1}\phi\right)  _{n}\left(  0,y\right)  =\left(  N_{1}\eta\right)
_{n}\left(  0,y\right)  ,
\]
with initial condition $\zeta\left(  0\right)  =0.$
\end{lemma}

\begin{proof}
We need to solve
\begin{equation}
\left[  M_{1}\left(  \zeta\omega\right)  \right]  _{n}=\left(  N_{1}%
\eta\right)  _{n}-\left(  M_{1}\phi^{\ast}\right)  _{n},\text{ for }x=0.
\label{MN}%
\end{equation}
Since $\omega_{n}$ satisfies $\left(  F_{1}\omega\right)  _{n}=\left(
M_{1}\omega\right)  _{n}=0,$ we find that the equation $\left(  \ref{MN}%
\right)  $ has the form%
\[
\frac{1}{i}\zeta^{\prime}\left(  y\right)  \omega_{n}=\left(  N_{1}%
\eta\right)  _{n}-\left(  M_{1}\phi^{\ast}\right)  _{n},\text{ }x=0.
\]
This is a first order ODE for $\zeta$ and has a unique solution with the
initial condition $\zeta\left(  0\right)  =0.$

We remark that due to the condition $\eta_{n+1}\left(  x,y\right)  =\eta
_{n}\left(  x+\frac{1}{2\sqrt{2}},y\right)  ,$ the above argument gives us
same $\zeta$ for different $n\in\mathbb{N}.$
\end{proof}

\begin{lemma}
\label{FI2}Let $\left\{  \phi_{n}\right\}  $ be given by $\left(
\ref{fi1}\right)  .$ Define $\Phi_{n}=\left(  M_{1}\phi\right)  _{n}-\left(
N_{1}\eta\right)  _{n}.$ Suppose $T_{\theta}\eta=0$ and $\left(  F_{1}%
\phi\right)  _{n}=\left(  G_{1}\eta\right)  _{n}.$ Then
\[
\partial_{x}\Phi_{n}=\lambda^{-1}\frac{\theta_{n-1}}{\theta_{n}}\Phi
_{n+1}+\left[  \left(  \frac{\partial_{s}\theta_{n}}{\theta_{n}}%
+\frac{\partial_{t}\theta_{n-1}}{\theta_{n-1}}\right)  +\lambda+\lambda
^{-1}\right]  \Phi_{n}-\lambda\frac{\theta_{n}}{\theta_{n-1}}\Phi_{n-1}.
\]

\end{lemma}

\begin{proof}
We compute%
\begin{align*}
\partial_{x}\left(  M_{1}\phi\right)  _{n}  &  =-i\partial_{y}\partial_{x}%
\phi_{n}-\partial_{x}\left[  \left(  \frac{\partial_{s}\theta_{n}}{\theta_{n}%
}-\frac{\partial_{t}\theta_{n-1}}{\theta_{n-1}}\right)  \phi_{n}\right] \\
&  -\lambda^{-1}\partial_{x}\left(  \phi_{n+1}\frac{\theta_{n-1}}{\theta_{n}%
}-\phi_{n}\right)  -\lambda\partial_{x}\left(  \phi_{n-1}\frac{\theta_{n}%
}{\theta_{n-1}}-\phi_{n}\right) \\
&  =-i\partial_{y}\left[  \left(  \frac{\partial_{s}\theta_{n}}{\theta_{n}%
}+\frac{\partial_{t}\theta_{n-1}}{\theta_{n-1}}\right)  \phi_{n}+\lambda
^{-1}\left(  \phi_{n+1}\frac{\theta_{n-1}}{\theta_{n}}-\phi_{n}\right)
\right] \\
&  -i\partial_{y}\left[  -\lambda\left(  \phi_{n-1}\frac{\theta_{n}}%
{\theta_{n-1}}-\phi_{n}\right)  +\left(  G_{1}\eta\right)  _{n}\right]
-\partial_{x}\left[  \left(  \frac{\partial_{s}\theta_{n}}{\theta_{n}}%
-\frac{\partial_{t}\theta_{n-1}}{\theta_{n-1}}\right)  \phi_{n}\right] \\
&  -\lambda^{-1}\partial_{x}\left(  \frac{\theta_{n-1}}{\theta_{n}}\right)
\phi_{n+1}-\lambda^{-1}\frac{\theta_{n-1}}{\theta_{n}}\partial_{x}\phi
_{n+1}+\lambda^{-1}\partial_{x}\phi_{n}\\
&  -\lambda\partial_{x}\left(  \frac{\theta_{n}}{\theta_{n-1}}\right)
\phi_{n-1}-\lambda\frac{\theta_{n}}{\theta_{n-1}}\partial_{x}\phi
_{n-1}+\lambda\partial_{x}\phi_{n}.
\end{align*}
Plugging the identity
\begin{align*}
\partial_{x}\phi_{n}  &  =\left(  \frac{\partial_{s}\theta_{n}}{\theta_{n}%
}+\frac{\partial_{t}\theta_{n-1}}{\theta_{n-1}}\right)  \phi_{n}+\lambda
^{-1}\left(  \phi_{n+1}\frac{\theta_{n-1}}{\tau_{n}^{\prime}}-\phi_{n}\right)
-\lambda\left(  \phi_{n-1}\frac{\theta_{n}}{\theta_{n-1}}-\phi_{n}\right)
+\left(  G_{1}\eta\right)  _{n},\\
-i\partial_{y}\phi_{n}  &  =\left(  \frac{\partial_{s}\theta_{n}}{\theta_{n}%
}-\frac{\partial_{t}\theta_{n-1}}{\theta_{n-1}}\right)  \phi_{n}+\lambda
^{-1}\left(  \phi_{n+1}\frac{\theta_{n-1}}{\theta_{n}}-\phi_{n}\right)
+\lambda\left(  \phi_{n-1}\frac{\theta_{n}}{\theta_{n-1}}-\phi_{n}\right)
+\left(  M_{1}\phi\right)  _{n}%
\end{align*}
into $\partial_{x}\left(  M_{1}\phi\right)  _{n},$ we find that the
coefficient before $\phi_{n+1}$ is
\begin{align*}
&  -i\lambda^{-1}\partial_{y}\left(  \frac{\theta_{n-1}}{\theta_{n}}\right)
-\lambda^{-1}\partial_{x}\left(  \frac{\theta_{n-1}}{\theta_{n}}\right) \\
&  -i\lambda^{-1}\frac{\theta_{n-1}}{\theta_{n}}\left[  i\left(
\frac{\partial_{s}\theta_{n+1}}{\theta_{n+1}}-\frac{\partial_{t}\theta_{n}%
}{\theta_{n}}\right)  -i\left(  \lambda+\lambda^{-1}\right)  \right] \\
&  -\left(  \frac{\partial_{s}\theta_{n}}{\theta_{n}}-\frac{\partial_{t}%
\theta_{n-1}}{\theta_{n-1}}\right)  \lambda^{-1}\frac{\theta_{n-1}}{\theta
_{n}}\\
&  -i\left[  \left(  \frac{\partial_{s}\theta_{n}}{\theta_{n}}+\frac
{\partial_{t}\theta_{n-1}}{\theta_{n-1}}\right)  +\left(  \lambda-\lambda
^{-1}\right)  \right]  i\lambda^{-1}\frac{\theta_{n-1}}{\theta_{n}}\\
&  -\lambda^{-1}\frac{\theta_{n-1}}{\theta_{n}}\left[  \left(  \frac
{\partial_{s}\theta_{n+1}}{\theta_{n+1}}+\frac{\partial_{t}\theta_{n}}%
{\theta_{n}}\right)  +\left(  \lambda-\lambda^{-1}\right)  \right] \\
&  +\left(  \lambda+\lambda^{-1}\right)  \lambda^{-1}\frac{\theta_{n-1}%
}{\theta_{n}}.
\end{align*}
One could verify directly that this is equal to $0.$ Similarly the coefficient
before $\phi_{n-1}$ and $\phi_{n}$ is equal to $0.$ Hence we get%
\begin{align*}
\partial_{x}\left(  M_{1}\phi\right)  _{n}  &  =\lambda^{-1}\frac{\theta
_{n-1}}{\theta_{n}}\left(  M_{1}\phi\right)  _{n+1}+\left[  \left(
\frac{\partial_{s}\theta_{n}}{\theta_{n}}+\frac{\partial_{t}\theta_{n-1}%
}{\theta_{n-1}}\right)  +\lambda+\lambda^{-1}\right]  \left(  M_{1}%
\phi\right)  _{n}\\
&  -\lambda\frac{\theta_{n}}{\theta_{n-1}}\left(  M_{1}\phi\right)
_{n-1}-i\partial_{y}\left(  G_{1}\eta\right)  _{n}-\lambda^{-1}\frac
{\theta_{n-1}}{\theta_{n}}\left(  G_{1}\eta\right)  _{n+1}\\
&  -\lambda\frac{\theta_{n}}{\theta_{n-1}}\left(  G_{1}\eta\right)
_{n-1}+\left(  \lambda+\lambda^{-1}\right)  \left(  G_{1}\eta\right)  _{n}.
\end{align*}
After some tedious computations(alternatively, we can use the software
\textit{Mathematica }to verify this), we find that
\[
\partial_{x}\Phi_{n}=\lambda^{-1}\frac{\theta_{n-1}}{\theta_{n}}\Phi
_{n+1}+\left[  \left(  \frac{\partial_{s}\theta_{n}}{\theta_{n}}%
+\frac{\partial_{t}\theta_{n-1}}{\theta_{n-1}}\right)  +\lambda+\lambda
^{-1}\right]  \Phi_{n}-\lambda\frac{\theta_{n}}{\theta_{n-1}}\Phi_{n-1}.
\]
\end{proof}
For $y\in\mathbb{R},$ we define
\[
k\left(  y\right)  :=\int_{+\infty}^{0}\frac{g_{1}\left(  s,y\right)  B\left(
s,y\right)  }{W\left(  s,y\right)  P_{1}\left(  s,y\right)  }ds,
\]
where $B$ is defined in $\left(  \ref{2}\right)  .$

\begin{lemma}
\label{k}As $y\rightarrow+\infty,$ $k\left(  y\right)  \rightarrow0.$
\end{lemma}

\begin{proof}
Recall that under the notation of Lemma \ref{linear}, $\eta_{n}=\theta
_{n}\tilde{\eta}_{n}.$ We compute%
\begin{align*}
\left(  G_{1}\eta\right)  _{n} &  =\frac{\omega_{n}}{\theta_{n}}\partial
_{s}\eta_{n}+\frac{\omega_{n}}{\theta_{n-1}}\partial_{t}\eta_{n-1}+\left(
-\frac{\partial_{s}\omega_{n}}{\theta_{n}}-\lambda^{-1}\frac{\omega_{n}%
}{\theta_{n}}-\lambda\frac{\omega_{n-1}}{\theta_{n-1}}\right)  \eta_{n}\\
&  +\left(  \lambda\frac{\omega_{n+1}}{\theta_{n}}-\frac{\partial_{t}%
\omega_{n}}{\theta_{n-1}}+\frac{\lambda\omega_{n}}{\theta_{n-1}}\right)
\eta_{n-1}\\
&  =\frac{\omega_{n}}{\theta_{n}}\left(  \frac{1}{2}\partial_{x}\left(
\theta_{n}\tilde{\eta}_{n}\right)  +\frac{1}{2i}\partial_{y}\left(  \theta
_{n}\tilde{\eta}_{n}\right)  \right)  +\frac{\omega_{n}}{\theta_{n-1}}\left(
\frac{1}{2}\partial_{x}\left(  \theta_{n-1}\tilde{\eta}_{n-1}\right)
-\frac{1}{2i}\partial_{y}\left(  \theta_{n-1}\tilde{\eta}_{n-1}\right)
\right)  \\
&  -\left(  \frac{\lambda}{\theta_{n}}+\lambda^{-1}\frac{\omega_{n}}%
{\theta_{n}}+\lambda\frac{\omega_{n-1}}{\theta_{n-1}}\right)  \theta_{n}%
\tilde{\eta}_{n}+\left(  \lambda\frac{\omega_{n+1}}{\theta_{n}}-\frac
{\lambda^{-1}}{\theta_{n-1}}+\frac{\lambda\omega_{n}}{\theta_{n-1}}\right)
\theta_{n-1}\tilde{\eta}_{n-1}\\
&  =\frac{\omega_{n}}{2}\partial_{x}\tilde{\eta}_{n}+\frac{1}{2}\omega
_{n}\frac{\partial_{x}\theta_{n}}{\theta_{n}}\tilde{\eta}_{n}+\frac{1}%
{2i}\omega_{n}\partial_{y}\tilde{\eta}_{n}+\frac{1}{2i}\frac{\omega
_{n}\partial_{y}\theta_{n}}{\theta_{n}}\tilde{\eta}_{n}\\
&  +\frac{\omega_{n}}{2}\partial_{x}\tilde{\eta}_{n-1}+\frac{1}{2}\frac
{\omega_{n}\partial_{x}\theta_{n-1}}{\theta_{n-1}}\tilde{\eta}_{n-1}-\frac
{1}{2i}\omega_{n}\partial_{y}\tilde{\eta}_{n-1}-\frac{1}{2i}\frac{\omega
_{n}\partial_{y}\theta_{n-1}}{\theta_{n-1}}\tilde{\eta}_{n-1}\\
&  -\left(  \lambda+\lambda^{-1}\omega_{n}+\lambda\frac{\omega_{n-1}\theta
_{n}}{\theta_{n-1}}\right)  \tilde{\eta}_{n}+\left(  \lambda\frac{\omega
_{n+1}}{\theta_{n}}\theta_{n-1}-\lambda^{-1}+\lambda\omega_{n}\right)
\tilde{\eta}_{n-1}.
\end{align*}
On the other hand, since $\tilde{\eta}_{n}$ satisfies $\left(  \ref{y}\right)
\ $and $\mathcal{F}\left(  f_{n}\right)  \left(  0,0\right)  =0,$ we infer
that
\[
\left\vert \left(  G_{1}\eta\right)  _{n}^{\symbol{94}}\right\vert \leq C.
\]
From this estimate and the fact that $W\left(  s,y\right)  =O\left(
s^{3}\right)  ,$ we get%
\begin{equation}
\left\vert \frac{g_{1}\left(  s,y\right)  B\left(  s,y\right)  }{W\left(
s,y\right)  P_{1}\left(  s,y\right)  }\right\vert \leq Cs^{2},\text{ for }%
s\in\left(  0,1\right)  ,\text{ }\label{B1}%
\end{equation}
and
\begin{equation}
\left\vert \frac{g_{2}\left(  s,y\right)  B\left(  s,y\right)  }{W\left(
s,y\right)  P_{1}\left(  s,y\right)  }\right\vert \leq Ce^{-2\pi\beta
s},\text{ for all }s>0.\label{B2}%
\end{equation}
It follows that
\[
\left\vert k\left(  y\right)  \right\vert \leq C\int_{0}^{\sqrt{\frac{1}{y}}%
}s^{2}ds+C\int_{\sqrt{\frac{1}{y}}}^{+\infty}e^{-2\pi\beta s}ds\rightarrow
0,\text{ as }y\rightarrow+\infty.
\]

\end{proof}

\begin{proof}
[Proof of Proposition \ref{P2}]For each fixed $y,$ using Lemma \ref{FI2},
$\Phi_{n}\left(  0,y\right)  =0,$ and the Gronwall inequality, we deduce
$\Phi_{n}\left(  x,y\right)  =0$ for all $x.$ Hence $\left\{  \phi
_{n}\right\}  $ solves the system $\left(  \ref{s2}\right)  .$ It remains to
prove the growth estimate for $\phi_{0}.$

Let us define
\[
\chi\left(  \xi,y\right)  =-\frac{1}{\left(  1-\gamma e^{2\pi ic\xi}\right)
e^{2\pi i\left(  \alpha+\beta i\right)  \left(  -\xi\right)  }}\left(
\frac{g_{1}^{\prime}}{J}\right)  ^{\prime}.
\]
Note that by the asymptotic behavior of $g_{1},$
\begin{equation}
\chi\left(  \xi,y\right)  =O\left(  \xi^{-5}\right)  . \label{ch}%
\end{equation}
Near $\xi=0,$ we can write
\[
\phi_{n}^{\symbol{94}}\left(  \xi,y\right)  =k\left(  y\right)  \chi\left(
\xi,y\right)  +O\left(  \xi^{-2}\right)  +\zeta\left(  y\right)  \omega
_{n}^{\symbol{94}}.
\]
Inserting this into the equation $\left[  \left(  M_{1}\phi\right)
_{n}\right]  ^{\symbol{94}}=\left[  \left(  N_{1}\eta\right)  _{n}\right]
^{\symbol{94}},$ using the asymptotic behavior $\left(  \ref{ch}\right)  $ of
$\chi,$ we find that $k^{\prime}\left(  y\right)  =0.$ Applying Lemma \ref{k},
we then deduce $k\left(  y\right)  =0.$ This in turn implies that
\[
g^{\ast}=-g_{2}\int_{0}^{\xi}\frac{g_{1}B}{WP_{1}}ds-g_{1}\int_{+\infty}^{\xi
}\frac{g_{2}B}{WP_{1}}ds.
\]
This together with $\left(  \ref{B1}\right)  $ and $\left(  \ref{B2}\right)  $
in particular tells us that
\begin{equation}
\left\vert \phi_{0}^{\symbol{94}}\left(  \xi,y\right)  -\zeta\left(  y\right)
\omega_{0}^{\symbol{94}}\right\vert \leq Ce^{2\pi\beta\left\vert
\xi\right\vert }\left(  1+\xi^{-2}\right)  . \label{fn}%
\end{equation}
After some manipulation on the Fourier integral representation of $\phi_{0},$
we obtain the desired estimate $\left(  \ref{f0}\right)  $ for $\phi_{0}.$
Indeed, the exponent $\frac{5}{8}$ can be replaced by any number larger than $\frac{1}{2}$(But in
general can not be $\frac{1}{2}$).
\end{proof}

\section{The linearized B\"acklund transformation between $\kappa$ and
$\omega$}

Linearizing the B\"{a}cklund transformation $\left(  \ref{b1}\right)  $ at
$\left\{  \kappa_{n}\right\}  ,\left\{  \omega_{n}\right\},$ we get
\begin{equation}
\left\{
\begin{array}
[c]{l}%
\partial_{s}\sigma_{n}\omega_{n}-\sigma_{n}\partial_{s}\omega_{n}%
-\lambda\left(  \sigma_{n+1}\omega_{n-1}-\sigma_{n}\omega_{n}\right) \\
=-\partial_{s}\kappa_{n}\phi_{n}+\kappa_{n}\partial_{s}\phi_{n}+\lambda\left(
\kappa_{n+1}\phi_{n-1}-\kappa_{n}\phi_{n}\right)  ,\\
\partial_{t}\sigma_{n+1}\omega_{n}-\sigma_{n+1}\partial_{t}\omega_{n}%
+\lambda^{-1}\left(  \sigma_{n}\omega_{n+1}-\sigma_{n+1}\omega_{n}\right) \\
=-\partial_{t}\kappa_{n+1}\phi_{n}+\kappa_{n+1}\partial_{t}\phi_{n}%
-\lambda^{-1}\left(  \kappa_{n}\phi_{n+1}-\kappa_{n+1}\phi_{n}\right)  .
\end{array}
\right.  \label{l1}%
\end{equation}
We write this system as
\begin{equation}
\left\{
\begin{array}
[c]{c}%
\left(  F_{0}\sigma\right)  _{n}=\left(  G_{0}\phi\right)  _{n},\\
\left(  M_{0}\sigma\right)  _{n}=\left(  N_{0}\phi\right)  _{n}.
\end{array}
\right.  \label{s1}%
\end{equation}
Here
\[
\left(  F_{0}\sigma\right)  _{n}=\partial_{x}\sigma_{n}-\lambda\frac
{\omega_{n-1}}{\omega_{n}}\left(  \sigma_{n+1}-\sigma_{n}\right)
-\lambda^{-1}\frac{\omega_{n}}{\omega_{n-1}}\left(  \sigma_{n}-\sigma
_{n-1}\right)  ,
\]%
\[
\left(  M_{0}\sigma\right)  _{n}=-i\partial_{y}\sigma_{n}-\lambda\frac
{\omega_{n-1}}{\omega_{n}}\left(  \sigma_{n+1}-\sigma_{n}\right)
+\lambda^{-1}\frac{\omega_{n}}{\omega_{n-1}}\left(  \sigma_{n}-\sigma
_{n-1}\right)  ,
\]
and%
\begin{align*}
\left(  G_{0}\phi\right)  _{n}  &  =\frac{-\partial_{s}\kappa_{n}\phi
_{n}+\kappa_{n}\partial_{s}\phi_{n}+\lambda\left(  \kappa_{n+1}\phi
_{n-1}-\kappa_{n}\phi_{n}\right)  }{\omega_{n}}\\
&  +\frac{-\partial_{t}\kappa_{n}\phi_{n-1}+\kappa_{n}\partial_{t}\phi
_{n-1}-\lambda^{-1}\left(  \kappa_{n-1}\phi_{n}-\kappa_{n}\phi_{n-1}\right)
}{\omega_{n-1}},
\end{align*}%
\begin{align*}
\left(  N_{0}\phi\right)  _{n}  &  =\frac{-\partial_{s}\kappa_{n}\phi
_{n}+\kappa_{n}\partial_{s}\phi_{n}+\lambda\left(  \kappa_{n+1}\phi
_{n-1}-\kappa_{n}\phi_{n}\right)  }{\omega_{n}}\\
&  -\frac{-\partial_{t}\kappa_{n}\phi_{n-1}+\kappa_{n}\partial_{t}\phi
_{n-1}-\lambda^{-1}\left(  \kappa_{n-1}\phi_{n}-\kappa_{n}\phi_{n-1}\right)
}{\omega_{n-1}}.
\end{align*}
The main result of this section is the following

\begin{proposition}
\label{P1}Let $\left\{  \phi_{n}\right\}  $ be given by Proposition \ref{P2}.
Then the system $\left(  \ref{s1}\right)  $ has a solution $\left\{
\sigma_{n}\right\}  $ with $\sigma_{n+1}\left(  x,y\right)  =\sigma_{n}\left(
x+\frac{1}{2\sqrt{2}},y\right)  $ and
\[
\left\vert \sigma_{0}\left(  x,y\right)  \right\vert \leq C\left(
1+x^{2}+y^{2}\right)  ^{\frac{5}{8}}.
\]

\end{proposition}

Let us define
\begin{equation}
P_{0}\left(  \xi\right)  =2\pi i\xi-\lambda\left(  e^{\frac{\pi i}{\sqrt{2}%
}\xi}-1\right)  -\lambda^{-1}\left(  1-e^{-\frac{\pi i}{\sqrt{2}}\xi}\right)
, \label{P}%
\end{equation}%
\[
Q_{0}\left(  \xi\right)  =-\lambda\frac{\pi i}{\sqrt{2}}\left(  e^{\frac{\pi
i}{\sqrt{2}}\xi}-1\right)  +\lambda^{-1}\frac{\pi i}{\sqrt{2}}\left(
1-e^{-\frac{\pi i}{\sqrt{2}}\xi}\right)  .
\]
We see from Lemma \ref{f1} that the Fourier transform of $\frac{1}{\omega_{n}%
}$ will depend on the sign of $y.$ Define
\[
\alpha_{0}=\frac{n+\frac{\sqrt{2}-1}{2}}{2\sqrt{2}},\text{ }\beta_{0}%
=\frac{\left\vert y\right\vert }{\sqrt{2}},
\]
and introduce a new function
\[
h_{1}\left(  \xi\right)  =\int_{-\infty}^{\xi}e^{2\pi i\left(  \alpha
_{0}+\beta_{0}i\right)  \left(  -s\right)  }\left(  e^{\frac{\pi i}{\sqrt{2}%
}s}-1\right)  \sigma_{n}^{\symbol{94}}\left(  s\right)  ds.
\]

\begin{lemma}
\label{h1}Suppose $y>0.$ Then
\begin{equation}
\left[  \left(  F_{0}\sigma\right)  _{n}\right]  \symbol{94}=\frac{e^{2\pi
i\left(  \alpha_{0}+\beta_{0}i\right)  \xi}}{e^{\frac{\pi i}{\sqrt{2}}\xi}%
-1}\left(  P_{0}\left(  \xi\right)  h_{1}^{\prime}\left(  \xi\right)
+Q_{0}\left(  \xi\right)  h_{1}\left(  \xi\right)  \right)  . \label{F0}%
\end{equation}

\end{lemma}

\begin{proof}
Using Lemma \ref{f1}, we get%
\begin{align*}
\left(  \frac{\omega_{n-1}}{\omega_{n}}\right)  ^{\symbol{94}}  &
=1^{\symbol{94}}-\left[  \frac{1}{2\sqrt{2}x+n+\frac{\sqrt{2}-1}{2}%
+2yi}\right]  ^{\symbol{94}}\\
&  =\delta-\frac{1}{2\sqrt{2}}\left(  -2\pi i\right)  e^{2\pi i\left(
\alpha_{0}+\beta_{0}i\right)  \xi}u\left(  \xi\right)  .
\end{align*}
Then we compute
\begin{align*}
\left[  \frac{\omega_{n-1}}{\omega_{n}}\left(  \sigma_{n+1}-\sigma_{n}\right)
\right]  ^{\symbol{94}}  &  =\left(  e^{\frac{\pi i}{\sqrt{2}}\xi}-1\right)
\sigma_{n}^{\symbol{94}}+\frac{\pi i}{\sqrt{2}}\left[  e^{2\pi i\left(
\alpha_{0}+\beta_{0}i\right)  \xi}u\left(  \xi\right)  \right]  \ast\left[
\left(  e^{\frac{\pi i}{\sqrt{2}}\xi}-1\right)  \sigma_{n}^{\symbol{94}%
}\right] \\
&  =\left(  e^{\frac{\pi i}{\sqrt{2}}\xi}-1\right)  \sigma_{n}^{\symbol{94}%
}+\frac{\pi i}{\sqrt{2}}\int_{-\infty}^{\xi}e^{2\pi i\left(  \alpha_{0}%
+\beta_{0}i\right)  \left(  \xi-s\right)  }\left(  e^{\frac{\pi i}{\sqrt{2}}%
s}-1\right)  \sigma_{n}^{\symbol{94}}\left(  s\right)  ds
\end{align*}
Similarly, we have
\begin{align*}
\left[  \frac{\omega_{n}}{\omega_{n-1}}\left(  \sigma_{n}-\sigma_{n-1}\right)
\right]  ^{\symbol{94}}  &  =\left(  1-e^{-\frac{\pi i}{\sqrt{2}}\xi}\right)
\sigma_{n}^{\symbol{94}}\\
&  +\frac{1}{2\sqrt{2}}\left(  -2\pi i\right)  \left[  e^{2\pi i\left(
\alpha_{0}+\beta_{0}i\right)  \xi}e^{-\frac{\pi i}{\sqrt{2}}\xi}u\left(
\xi\right)  \right]  \ast\left[  \left(  1-e^{-\frac{\pi i}{\sqrt{2}}\xi
}\right)  \sigma_{n}^{\symbol{94}}\right] \\
&  =\left(  1-e^{-\frac{\pi i}{\sqrt{2}}\xi}\right)  \sigma_{n}^{\symbol{94}%
}\\
&  -\frac{\pi i}{\sqrt{2}}\int_{-\infty}^{\xi}e^{2\pi i\left(  \alpha
_{0}+\beta_{0}i\right)  \left(  \xi-s\right)  }e^{-\frac{\pi i}{\sqrt{2}%
}\left(  \xi-s\right)  }\left(  1-e^{-\frac{\pi i}{\sqrt{2}}s}\right)
\sigma_{n}^{\symbol{94}}\left(  s\right)  ds.
\end{align*}
Consequently,
\begin{align*}
\left[  \left(  F_{0}\sigma\right)  _{n}\right]  \symbol{94}  &  =2\pi
i\xi\sigma_{n}^{\symbol{94}}-\lambda\left(  e^{\frac{\pi i}{\sqrt{2}}\xi
}-1\right)  \sigma_{n}^{\symbol{94}}\\
&  -\lambda\frac{\pi i}{\sqrt{2}}\int_{-\infty}^{\xi}e^{2\pi i\left(
\alpha_{0}+\beta_{0}i\right)  \left(  \xi-s\right)  }\left(  e^{\frac{\pi
i}{\sqrt{2}}s}-1\right)  \sigma_{n}^{\symbol{94}}\left(  s\right)  ds\\
&  -\lambda^{-1}\left(  1-e^{-\frac{\pi i}{\sqrt{2}}\xi}\right)  \sigma
_{n}^{\symbol{94}}\\
&  +\lambda^{-1}\frac{\pi i}{\sqrt{2}}\int_{-\infty}^{\xi}e^{2\pi i\left(
\alpha_{0}+\beta_{0}i\right)  \left(  \xi-s\right)  }e^{-\frac{\pi i}{\sqrt
{2}}\left(  \xi-s\right)  }\left(  1-e^{-\frac{\pi i}{\sqrt{2}}s}\right)
\sigma_{n}^{\symbol{94}}\left(  s\right)  ds\\
&  =2\pi i\xi\frac{h_{1}^{\prime}\left(  \xi\right)  e^{2\pi i\left(
\alpha_{0}+\beta_{0}i\right)  \xi}}{\left(  e^{\frac{\pi i}{\sqrt{2}}\xi
}-1\right)  }-\lambda h_{1}^{\prime}\left(  \xi\right)  e^{2\pi i\left(
\alpha_{0}+\beta_{0}i\right)  \xi}-\lambda\frac{\pi i}{\sqrt{2}}e^{2\pi
i\left(  \alpha_{0}+\beta_{0}i\right)  \xi}h_{1}\\
&  -\lambda^{-1}e^{-\frac{\pi i}{\sqrt{2}}\xi}h_{1}^{\prime}\left(
\xi\right)  e^{2\pi i\left(  \alpha_{0}+\beta_{0}i\right)  \xi}+\lambda
^{-1}\frac{\pi i}{\sqrt{2}}e^{-\frac{\pi i}{\sqrt{2}}\xi}e^{2\pi i\left(
\alpha_{0}+\beta_{0}i\right)  \xi}h_{1}.
\end{align*}
This is $\left(  \ref{F0}\right)  .$
\end{proof}

Now let us define
\[
h_{2}\left(  \xi\right)  =\int_{\xi}^{+\infty}e^{2\pi i\left(  \alpha
_{0}-\beta_{0}i\right)  \left(  -s\right)  }\left(  e^{\frac{\pi i}{\sqrt{2}%
}s}-1\right)  \sigma_{n}^{\symbol{94}}\left(  s\right)  ds.
\]

\begin{lemma}
\label{h-1}Suppose $y<0.$ Then
\[
\left[  \left(  F_{0}\sigma\right)  _{n}\right]  \symbol{94}=-\frac{e^{2\pi
i\left(  \alpha_{0}-\beta_{0}i\right)  \xi}}{e^{\frac{\pi i}{\sqrt{2}}\xi}%
-1}\left(  P\left(  \xi\right)  h_{2}^{\prime}\left(  \xi\right)  +Q\left(
\xi\right)  h_{2}\left(  \xi\right)  \right)  .
\]

\end{lemma}

\begin{proof}
The proof is similar to that in the previous lemma. By Lemma \ref{f2},%
\begin{align*}
\left(  \frac{\omega_{n-1}}{\omega_{n}}\right)  ^{\symbol{94}}  &
=1^{\symbol{94}}-\left[  \frac{1}{2\sqrt{2}x+n+\frac{\sqrt{2}-1}{2}-\left\vert
y\right\vert i}\right]  ^{\symbol{94}}\\
&  =\delta-\frac{1}{2\sqrt{2}}\left(  2\pi i\right)  e^{2\pi i\left(
\alpha_{0}-\beta_{0}i\right)  \xi}u\left(  -\xi\right)  .
\end{align*}
Then we compute
\begin{align*}
\left[  \frac{\omega_{n-1}}{\omega_{n}}\left(  \sigma_{n+1}-\sigma_{n}\right)
\right]  ^{\symbol{94}}  &  =\left(  e^{\frac{\pi i}{\sqrt{2}}\xi}-1\right)
\sigma_{n}^{\symbol{94}}-\frac{\pi i}{\sqrt{2}}\left[  e^{2\pi i\left(
\alpha_{0}-\beta_{0}i\right)  \xi}u\left(  -\xi\right)  \right]  \ast\left[
\left(  e^{\frac{\pi i}{\sqrt{2}}\xi}-1\right)  \sigma_{n}^{\symbol{94}%
}\right] \\
&  =\left(  e^{\frac{\pi i}{\sqrt{2}}\xi}-1\right)  \sigma_{n}^{\symbol{94}%
}-\frac{\pi i}{\sqrt{2}}\int_{\xi}^{+\infty}e^{2\pi i\left(  \alpha_{0}%
-\beta_{0}i\right)  \left(  \xi-s\right)  }\left(  e^{\frac{\pi i}{\sqrt{2}}%
s}-1\right)  \sigma_{n}^{\symbol{94}}\left(  s\right)  ds.
\end{align*}
We also have
\begin{align*}
\left[  \frac{\omega_{n}}{\omega_{n-1}}\left(  \sigma_{n}-\sigma_{n-1}\right)
\right]  ^{\symbol{94}}  &  =\left(  1-e^{-\frac{\pi i}{\sqrt{2}}\xi}\right)
\sigma_{n}^{\symbol{94}}\\
&  +\frac{1}{2\sqrt{2}}\left(  2\pi i\right)  \left[  e^{2\pi i\left(
\alpha_{0}-\beta_{0}i\right)  \xi}e^{-\frac{\pi i}{\sqrt{2}}\xi}u\left(
-\xi\right)  \right]  \ast\left[  \left(  1-e^{-\frac{\pi i}{\sqrt{2}}\xi
}\right)  \sigma_{n}^{\symbol{94}}\right] \\
&  =\left(  1-e^{-\frac{\pi i}{\sqrt{2}}\xi}\right)  \sigma_{n}^{\symbol{94}%
}\\
&  +\frac{\pi i}{\sqrt{2}}\int_{\xi}^{+\infty}e^{2\pi i\left(  \alpha
_{0}-\beta_{0}i\right)  \left(  \xi-s\right)  }e^{-\frac{\pi i}{\sqrt{2}%
}\left(  \xi-s\right)  }\left(  1-e^{-\frac{\pi i}{\sqrt{2}}s}\right)
\sigma_{n}^{\symbol{94}}\left(  s\right)  ds.
\end{align*}
It follows that
\begin{align*}
&  \left[  \left(  F_{0}\sigma\right)  _{n}\right]  \symbol{94}=2\pi
i\xi\sigma_{n}^{\symbol{94}}-\lambda\left(  e^{\frac{\pi i}{\sqrt{2}}\xi
}-1\right)  \sigma_{n}^{\symbol{94}}\\
&  +\lambda\frac{\pi i}{\sqrt{2}}\int_{-\infty}^{\xi}e^{2\pi i\left(
\alpha_{0}+\beta_{0}i\right)  \left(  \xi-s\right)  }\left(  e^{\frac{\pi
i}{\sqrt{2}}s}-1\right)  \sigma_{n}^{\symbol{94}}\left(  s\right)
ds-\lambda^{-1}\left(  1-e^{-\frac{\pi i}{\sqrt{2}}\xi}\right)  \sigma
_{n}^{\symbol{94}}\\
&  -\lambda^{-1}\frac{\pi i}{\sqrt{2}}\int_{-\infty}^{\xi}e^{2\pi i\left(
\alpha_{0}+\beta_{0}i\right)  \left(  \xi-s\right)  }e^{-\frac{\pi i}{\sqrt
{2}}\left(  \xi-s\right)  }\left(  1-e^{-\frac{\pi i}{\sqrt{2}}s}\right)
\sigma_{n}^{\symbol{94}}\left(  s\right)  ds\\
&  =-2\pi i\xi\frac{h_{2}^{\prime}e^{2\pi i\left(  \alpha_{0}-\beta
_{0}i\right)  \xi}}{\left(  e^{\frac{\pi i}{\sqrt{2}}\xi}-1\right)  }+\lambda
h_{2}^{\prime}e^{2\pi i\left(  \alpha_{0}-\beta_{0}i\right)  \xi}\\
&  +\lambda\frac{\pi i}{\sqrt{2}}e^{2\pi i\left(  \alpha_{0}-\beta
_{0}i\right)  \xi}h_{2}+\lambda^{-1}e^{-\frac{\pi i}{\sqrt{2}}\xi}%
h_{2}^{\prime}e^{2\pi i\left(  \alpha_{0}-\beta_{0}i\right)  \xi}-\lambda
^{-1}\frac{\pi i}{\sqrt{2}}e^{-\frac{\pi i}{\sqrt{2}}\xi}e^{2\pi i\left(
\alpha_{0}-\beta_{0}i\right)  \xi}h_{2}.
\end{align*}
We conclude that
\begin{align*}
e^{-2\pi i\left(  \alpha_{0}-\beta_{0}i\right)  \xi}\left[  \left(
F_{0}\sigma\right)  _{n}\right]  \symbol{94}  &  =\left[  \frac{-2\pi i\xi
}{\left(  e^{\frac{\pi i}{\sqrt{2}}\xi}-1\right)  }+\lambda+\lambda
^{-1}e^{-\frac{\pi i}{\sqrt{2}}\xi}\right]  h_{2}^{\prime}\\
&  +\left[  \lambda\frac{\pi i}{\sqrt{2}}-\lambda^{-1}\frac{\pi i}{\sqrt{2}%
}e^{-\frac{\pi i}{\sqrt{2}}\xi}\right]  h_{2}.
\end{align*}
The proof is completed.
\end{proof}

\begin{lemma}
\label{Asy}\bigskip The functions $P_{0}$ and $Q_{0}$ have the following
asymptotic behavior as $\xi\rightarrow0:$%
\begin{align*}
P_{0}\left(  \xi\right)   &  =\frac{\pi^{2}\xi^{2}}{2}+O\left(  \xi
^{3}\right)  ,\text{ }\\
Q_{0}\left(  \xi\right)   &  =\pi^{2}\xi+O\left(  \xi^{2}\right)  .
\end{align*}
Moreover,
\[
P_{0}\left(  \xi\right)  \neq0,\text{for }\xi\neq0.
\]

\end{lemma}

\begin{proof}
This follows from direct computations.
\end{proof}

\begin{lemma}
\label{ODE}\bigskip The equation
\[
P_{0}\left(  \xi\right)  h^{\prime}\left(  \xi\right)  +Q_{0}\left(
\xi\right)  h\left(  \xi\right)  =0
\]
has a solution $\rho$ such that%
\[
\rho\left(  \xi\right)  =\xi^{-2}+O\left(  \xi^{-1}\right)  \text{, as }%
\xi\rightarrow0.
\]

\end{lemma}

\begin{proof}
This follows from the asymptotic behavior of $P_{0}$ and $Q_{0}$ near $0.$
\end{proof}

Having understood the Fourier transform of the $\left(  F_{0}\sigma\right)
_{n},$ we proceed to solve the equation%
\[
\left(  F_{0}\sigma\right)  _{n}=\left(  G_{0}\sigma\right)  _{n}.
\]
Taking Fourier transform, we get $\left[  \left(  F_{0}\sigma\right)
_{n}\right]  ^{\symbol{94}}=\left[  \left(  G_{0}\sigma\right)  _{n}\right]
^{\symbol{94}}.$ Using Lemma \ref{h1} and Lemma \ref{h-1}, we arrive at a
first order ODE. Let
\[
A\left(  \xi,y\right)  =\frac{e^{\frac{\pi i}{\sqrt{2}}\xi}-1}{e^{2\pi
i\left(  \alpha_{0}+\frac{yi}{\sqrt{2}}\right)  \xi}P_{0}}\left[  \left(
G_{0}\phi\right)  _{n}\right]  ^{\symbol{94}}.
\]

For $y>0,$ we define
\[
h^{\ast}\left(  \xi,y\right)  =\rho\left(  \xi\right)  \int_{-\infty}^{\xi
}\left(  \rho\left(  s\right)  \right)  ^{-1}A\left(  \xi,y\right)  ds,
\]
and
\[
\sigma_{n}^{\ast\symbol{94}}\left(  \xi,y\right)  =\frac{e^{2\pi i\left(
\alpha_{0}+\frac{yi}{\sqrt{2}}\right)  \xi}h^{\ast\prime}}{\ e^{\frac{\pi
i}{\sqrt{2}}\xi}-1}=-\frac{e^{2\pi i\left(  \alpha_{0}+\frac{yi}{\sqrt{2}%
}\right)  \xi}h^{\ast}Q_{0}}{\left(  e^{\frac{\pi i}{\sqrt{2}}\xi}-1\right)
P_{0}}.
\]
Similarly, if $y<0,$ we define
\[
h^{\ast}\left(  \xi,y\right)  =\rho\left(  \xi\right)  \int_{+\infty}^{\xi
}\left(  \rho\left(  s\right)  \right)  ^{-1}A\left(  \xi,y\right)  ds,
\]
and
\[
\sigma_{n}^{\ast\symbol{94}}\left(  \xi,y\right)  =\frac{e^{2\pi i\left(
\alpha_{0}+\frac{yi}{\sqrt{2}}\right)  \xi}h^{\ast\prime}}{\ e^{\frac{\pi
i}{\sqrt{2}}\xi}-1}=-\frac{e^{2\pi i\left(  \alpha_{0}+\frac{yi}{\sqrt{2}%
}\right)  \xi}h^{\ast}Q_{0}}{\left(  e^{\frac{\pi i}{\sqrt{2}}\xi}-1\right)
P_{0}}.
\]

We would like to solve the equation $\left(  M_{0}\sigma\right)  _{n}=\left(
N_{0}\phi\right)  _{n}.$

\begin{lemma}
The equation
\[
\left(  M_{0}\sigma\right)  _{n}\left(  0,y\right)  =\left(  N_{0}\phi\right)
_{n}\left(  0,y\right)  ,y>0.
\]
has a solution of the form $\sigma_{n}^{\ast}+\gamma\left(  y\right)  ,$ with
the initial condition $\gamma\left(  0\right)  =0.$
\end{lemma}

\begin{proof}
This is a first order ODE for the function $\gamma,$ of the form
\[
\frac{1}{i}\gamma^{\prime}\left(  y\right)  =\left(  N_{0}\phi\right)
_{n}\left(  0,y\right)  -\left(  M_{0}\sigma^{\ast}\right)  _{n}\left(
0,y\right)  .
\]
Integrating this equation, we get the solution.
\end{proof}

Similarly, we have

\begin{lemma}
The equation
\[
\left(  M_{0}\sigma\right)  _{n}\left(  0,y\right)  =\left(  N_{0}\phi\right)
_{n}\left(  0,y\right)  ,y<0.
\]
has a solution $\sigma_{n}$ of the form $\sigma_{n}^{\ast}+\gamma\left(
y\right)  ,$ with the initial condition $\gamma\left(  0\right)  =0.$
\end{lemma}

To proceed, slightly abuse the notation with the previous section, we define
$\Phi_{n}=\left(  M_{0}\sigma\right)  _{n}-\left(  N_{0}\phi\right)  _{n}. $

\begin{lemma}
Assume $y\neq0.$ Suppose $T_{\omega}\phi=0$ and $\left(  F_{0}\sigma\right)
_{n}=\left(  G_{0}\phi\right)  _{n},n\in\mathbb{N}.$ Then
\begin{equation}
\partial_{x}\Phi_{n}=\lambda\frac{\omega_{n-1}}{\omega_{n}}\Phi_{n+1}+\left(
\lambda^{-1}\frac{\omega_{n}}{\omega_{n-1}}-\lambda\frac{\omega_{n-1}}%
{\omega_{n}}\right)  \Phi_{n}-\lambda^{-1}\frac{\omega_{n}}{\omega_{n-1}}%
\Phi_{n-1}. \label{FI}%
\end{equation}

\end{lemma}

\begin{proof}
We compute%
\begin{align*}
\partial_{x}\left(  M_{0}\sigma\right)  _{n}  &  =-i\partial_{y}\partial
_{x}\sigma_{n}-\lambda\partial_{x}\left[  \frac{\omega_{n-1}}{\omega_{n}%
}\left(  \sigma_{n+1}-\sigma_{n}\right)  \right]  +\lambda^{-1}\partial
_{x}\left[  \frac{\omega_{n}}{\omega_{n-1}}\left(  \sigma_{n}-\sigma
_{n-1}\right)  \right] \\
&  =-i\partial_{y}\left[  \lambda\frac{\omega_{n-1}}{\omega_{n}}\left(
\sigma_{n+1}-\sigma_{n}\right)  +\lambda^{-1}\frac{\omega_{n}}{\omega_{n-1}%
}\left(  \sigma_{n}-\sigma_{n-1}\right)  +\left(  G_{0}\phi\right)
_{n}\right] \\
&  -\lambda\partial_{x}\left(  \frac{\omega_{n-1}}{\omega_{n}}\right)  \left(
\sigma_{n+1}-\sigma_{n}\right)  -\lambda\frac{\omega_{n-1}}{\omega_{n}%
}\partial_{x}\left(  \sigma_{n+1}-\sigma_{n}\right) \\
&  +\lambda^{-1}\partial_{x}\left(  \frac{\omega_{n}}{\omega_{n-1}}\right)
\left(  \sigma_{n}-\sigma_{n-1}\right)  +\lambda^{-1}\frac{\omega_{n}}%
{\omega_{n-1}}\partial_{x}\left(  \sigma_{n}-\sigma_{n-1}\right)  .
\end{align*}
Inserting the identity
\begin{align*}
\partial_{x}\sigma_{n}  &  =\lambda\frac{\omega_{n-1}}{\omega_{n}}\left(
\sigma_{n+1}-\sigma_{n}\right)  +\lambda^{-1}\frac{\omega_{n}}{\omega_{n-1}%
}\left(  \sigma_{n}-\sigma_{n-1}\right)  +\left(  G_{0}\phi\right)  _{n},\\
\partial_{y}\sigma_{n}  &  =i\left[  -\lambda\frac{\omega_{n-1}}{\omega_{n}%
}\left(  \sigma_{n+1}-\sigma_{n}\right)  +\lambda^{-1}\frac{\omega_{n}}%
{\omega_{n-1}}\left(  \sigma_{n}-\sigma_{n-1}\right)  \right]  +i\left(
M_{0}\sigma\right)  _{n},
\end{align*}
into $\partial_{x}\left(  M_{0}\sigma\right)  _{n},$ we find that
\begin{align}
\partial_{x}\left(  M_{0}\sigma\right)  _{n}  &  =\lambda\frac{\omega_{n-1}%
}{\omega_{n}}\left(  M_{0}\sigma\right)  _{n+1}+\left(  \lambda^{-1}%
\frac{\omega_{n}}{\omega_{n-1}}-\lambda\frac{\omega_{n-1}}{\omega_{n}}\right)
\left(  M_{0}\sigma\right)  _{n}\nonumber\\
&  -\lambda^{-1}\frac{\omega_{n}}{\omega_{n-1}}\left(  M_{0}\sigma\right)
_{n-1}-i\partial_{y}\left(  G_{0}\phi\right)  _{n}\nonumber\\
&  -\lambda\frac{\omega_{n-1}}{\omega_{n}}\left(  \left(  G_{0}\phi\right)
_{n+1}-\left(  G_{0}\phi\right)  _{n}\right) \nonumber\\
&  +\lambda^{-1}\frac{\omega_{n}}{\omega_{n-1}}\left(  \left(  G_{0}%
\phi\right)  _{n}-\left(  G_{0}\phi\right)  _{n-1}\right)  . \label{mbar}%
\end{align}

On the other hand, we compute%
\begin{align*}
-i\partial_{y}\left(  G_{0}\phi\right)  _{n}-\partial_{x}\left(  N_{0}%
\phi\right)  _{n}  &  =-2\partial_{t}\left[  \frac{\partial_{s}\phi
_{n}+\lambda\left(  \phi_{n-1}-\phi_{n}\right)  }{\omega_{n}}\right]  \ \\
&  +2\partial_{s}\left[  \frac{\partial_{t}\phi_{n-1}-\lambda^{-1}\left(
\phi_{n}-\phi_{n-1}\right)  }{\omega_{n-1}}\right]  .
\end{align*}
Using the fact that $T_{\omega}\phi=0,$ we get
\begin{align*}
&  -i\partial_{y}\left(  G_{0}\phi\right)  _{n}-\partial_{x}\left(  N_{0}%
\phi\right)  _{n}\\
&  =-\lambda\frac{\omega_{n-1}}{\omega_{n}}\left(  N_{0}\phi\right)
_{n+1}-\left(  \lambda^{-1}\frac{\omega_{n}}{\omega_{n-1}}-\lambda\frac
{\omega_{n-1}}{\omega_{n}}\right)  \left(  N_{0}\phi\right)  _{n}+\lambda
^{-1}\frac{\omega_{n}}{\omega_{n-1}}\left(  N_{0}\phi\right)  _{n-1}\\
&  +\lambda\frac{\omega_{n-1}}{\omega_{n}}\left(  \left(  G_{0}\phi\right)
_{n+1}-\left(  G_{0}\phi\right)  _{n}\right)  -\lambda^{-1}\frac{\omega_{n}%
}{\omega_{n-1}}\left(  \left(  G_{0}\phi\right)  _{n}-\left(  G_{0}%
\phi\right)  _{n-1}\right)  .
\end{align*}
This identity together with $\left(  \ref{mbar}\right)  $ yield $\left(
\ref{FI}\right)  .$
\end{proof}

\begin{proof}
[Proof of Proposition \ref{P1}]For each fixed $y\neq0,$ since $\Phi_{n}$
satisfies $\left(  \ref{FI}\right)  $ and the initial condition $\Phi
_{n}\left(  0,y\right)  =0,$ we deduce that $\Phi_{n}\left(  x,y\right)  =0,$
for all $x\in\mathbb{R}.$ Observe that $\sigma_{n}$ may have a jump across the
$x$ axis.

We would like to show that actually $\sigma_{n}$ is continuous at $y=0,$ that
is,
\[
\lim_{y\rightarrow0^{+}}\sigma_{n}\left(  x,y\right)  =\lim_{y\rightarrow
0^{-}}\sigma_{n}\left(  x,y\right)  .
\]
To see this, according to the definition of $\sigma_{n}^{\ast},$ it will be
suffice to prove
\begin{equation}
\int_{-\infty}^{0}\rho^{-1}\left(  s\right)  A\left(  s,y\right)  ds=0,y>0,
\label{1}%
\end{equation}
and
\begin{equation}
\int_{+\infty}^{0}\rho^{-1}\left(  s\right)  A\left(  s,y\right)  ds=0,y<0.
\label{11}%
\end{equation}
Let $k_{1}\left(  y\right)  =\int_{-\infty}^{0}\rho^{-1}\left(  s\right)
A\left(  s,y\right)  ds,y>0.$ Using the estimate of the Fourier transform of
$\phi_{n}($See $\left(  \ref{fn}\right)  )$ and the asymptotic behavior of
$\rho,$ we can show that $k_{1}\left(  y\right)  \rightarrow0,$ as
$y\rightarrow+\infty.$ On the other hand, letting
\[
w\left(  \xi,y\right)  =-\frac{e^{2\pi i\left(  \alpha_{0}+\frac{yi}{\sqrt{2}%
}\right)  \xi}\rho^{\prime}Q_{0}}{\left(  e^{\frac{\pi i}{\sqrt{2}}\xi
}-1\right)  P_{0}},
\]
we have $\sigma_{n}^{\ast\symbol{94}}\left(  \xi,y\right)  =k_{1}\left(
y\right)  w\left(  \xi,y\right)  +O\left(  \xi^{-2}\right)  ,$ for $\xi$ close
to $0.$ Inserting this into the second equation of $\left(  \ref{s1}\right)  $
we conclude $k_{1}^{\prime}=0.$ Hence $k_{1}\left(  y\right)  =0.$ Similarly,
$\int_{+\infty}^{0}\rho^{-1}\left(  s\right)  A\left(  s,y\right)  ds=0$ for
$y<0.$

Once $\left(  \ref{1}\right)  $ and $\left(  \ref{11}\right)  $ have been
proved, we get
\[
\left\vert \sigma_{0}^{\symbol{94}}\left(  \xi,y\right)  -\gamma\left(
y\right)  \delta\right\vert \leq Ce^{2\pi\beta_{0}\left\vert \xi\right\vert
}\xi^{-2},\text{ for }\xi\text{ close to }0.
\]
Analysis of the Fourier integral of $\sigma_{0}$ then tells us that
\[
\left\vert \sigma_{0}\left(  x,y\right)  \right\vert \leq C\left(
1+x^{2}+y^{2}\right)  ^{\frac{5}{8}}.
\]
This finishes the proof.
\end{proof}

\section{Proof of Theorem \ref{Main}}

We have analyzed the linear B\"acklund transformation in the previous
sections. Based on this, we will prove our main theorem in this section. Let
us define
\[
\left(  F_{0}^{\ast}\sigma\right)  _{n}=\partial_{s}\sigma_{n}\omega
_{n}-\sigma_{n}\partial_{s}\omega_{n}-\lambda\left(  \sigma_{n+1}\omega
_{n-1}-\sigma_{n}\omega_{n}\right)
\]
and
\[
\left(  M_{0}^{\ast}\sigma\right)  _{n}=\partial_{t}\sigma_{n+1}\omega
_{n}-\sigma_{n+1}\partial_{t}\omega_{n}+\lambda^{-1}\left(  \sigma_{n}%
\omega_{n+1}-\sigma_{n+1}\omega_{n}\right)  .
\]

\begin{lemma}
\label{L3}Suppose $\left\{  \sigma_{n}\right\}  ,\left\{  \phi_{n}\right\}  $
satisfy the system $\left(  \ref{l1}\right)  $. Then
\begin{equation}
\partial_{x}\phi_{n}-2\phi_{n}+\lambda\phi_{n-1}-\lambda_{n+1}^{-1}%
\phi=\left(  F_{0}^{\ast}\sigma\right)  _{n}+\left(  M_{0}^{\ast}%
\sigma\right)  _{n}, \label{y1}%
\end{equation}
and
\begin{equation}
\frac{1}{i}\partial_{y}\phi_{n}-2\sqrt{2}\phi_{n}+\lambda\phi_{n-1}%
+\lambda_{n+1}^{-1}\phi=\left(  F_{0}^{\ast}\sigma\right)  _{n}-\left(
M_{0}^{\ast}\sigma\right)  _{n}. \label{y2}%
\end{equation}
In particular, if $F_{0}^{\ast}\sigma=M_{0}^{\ast}\sigma=0$, and
\[
\phi_{n+1}\left(  x,y\right)  =\phi_{n}\left(  x+\frac{1}{2\sqrt{2}},y\right)
,
\]%
\begin{equation}
\left\vert \phi_{0}\right\vert \leq C\left(  1+x^{2}+y^{2}\right)  ^{\frac
{5}{8}}, \label{es}%
\end{equation}
then $\phi_{n}=c_{1}+c_{2}\omega_{n}$ for some constants $c_{1},c_{2}.$
\end{lemma}

\begin{proof}
The equations $\left(  \ref{y1}\right)  $ and $\left(  \ref{y2}\right)  $ are
obtained from adding and subtracting the two equations in $\left(
\ref{l1}\right)  $. If $F_{0}^{\ast}\sigma=M_{0}^{\ast}\sigma=0,$ then by
$\left(  \ref{y1}\right)  $, we have
\[
\partial_{x}\phi_{n}-2\phi_{n}+\lambda\phi_{n-1}-\lambda^{-1}\phi_{n+1}=0.
\]
Fourier transform tells us that $\phi_{n}=a_{1}+a_{2}x+b\left(  y\right)  .$
Inserting this into the equation
\[
\frac{1}{i}\partial_{y}\phi_{n}-2\sqrt{2}\phi_{n}+\lambda\phi_{n-1}%
+\lambda_{n+1}^{-1}\phi_{n-1}=0,
\]
using the estimate $\left(  \ref{es}\right)  ,$ we find that $\phi_{n}%
=c_{1}+c_{2}\omega_{n}$ for some constants $c_{1},c_{2}.$ This finishes the proof.
\end{proof}

\begin{lemma}
\label{M1}%
\[
\left\{
\begin{array}
[c]{c}%
F_{0}\left(  2\sqrt{2}x+n\right)  =G_{0}\left(  \left(  2\sqrt{2}x+n\right)
^{2}+2yi\left(  2\sqrt{2}x+n\right)  +\left(  1-\sqrt{2}\right)  yi\right)
,\\
M_{0}\left(  2\sqrt{2}x+n\right)  =N_{0}\left(  \left(  2\sqrt{2}x+n\right)
^{2}+2yi\left(  2\sqrt{2}x+n\right)  +\left(  1-\sqrt{2}\right)  yi\right)  .
\end{array}
\right.
\]%
\[
\left\{
\begin{array}
[c]{c}%
F_{0}\left(  y\right)  =G_{0}\left(  \left(  2\sqrt{2}x+n\right)
y+\frac{\sqrt{2}-1}{2}y\right)  ,\\
M_{0}\left(  y\right)  =N_{0}\left(  \left(  2\sqrt{2}x+n\right)
y+\frac{\sqrt{2}-1}{2}y\right)  .
\end{array}
\right.
\]

\end{lemma}

\begin{proof}
We compute
\[
\left[  F_{0}^{\ast}\left(  2\sqrt{2}x+n\right)  \right]  _{n}=-2\sqrt
{2}x-2iy-n+\frac{\sqrt{2}+3}{2},
\]%
\[
\left[  M_{0}^{\ast}\left(  2\sqrt{2}x+n\right)  \right]  _{n}=n+2iy+2\sqrt
{2}x-\frac{1}{2}\sqrt{2}+\frac{1}{2}.
\]
Solving the equations
\[
\left\{
\begin{array}
[c]{c}%
\partial_{x}\phi_{n}-2\phi_{n}+\lambda\phi_{n-1}-\lambda_{n+1}^{-1}\phi
_{n+1}=\left[  F_{0}^{\ast}\left(  2\sqrt{2}x+n\right)  \right]  _{n}+\left[
M_{0}^{\ast}\left(  2\sqrt{2}x+n\right)  \right]  _{n},\\
\frac{1}{i}\partial_{y}\phi_{n}-2\sqrt{2}\phi_{n}+\lambda\phi_{n-1}%
+\lambda^{-1}\phi_{n+1}=\left[  F_{0}^{\ast}\left(  2\sqrt{2}x+n\right)
\right]  _{n}-\left[  M_{0}^{\ast}\left(  2\sqrt{2}x+n\right)  \right]  _{n},
\end{array}
\right.
\]
we get a solution%
\[
\left(  2\sqrt{2}x+n\right)  ^{2}+2yi\left(  2\sqrt{2}x+n\right)  +\left(
1-\sqrt{2}\right)  yi.
\]

Similarly,
\[
F_{0}^{\ast}\left(  y\right)  =\frac{1}{2i}\left(  2\sqrt{2}x+n+2yi+\frac
{\sqrt{2}-1}{2}\right)  ,
\]%
\[
M_{0}^{\ast}\left(  y\right)  =-\frac{1}{2i}\left(  2\sqrt{2}x+n+2yi+\frac
{\sqrt{2}-1}{2}\right)  .
\]
Solving the system
\[
\left\{
\begin{array}
[c]{l}%
\partial_{x}\phi_{n}-2\phi_{n}+\lambda\phi_{n-1}-\lambda_{n+1}^{-1}\phi
_{n+1}=0,\\
\frac{1}{i}\partial_{y}\phi_{n}-2\sqrt{2}\phi_{n}+\lambda\phi_{n-1}%
+\lambda^{-1}\phi_{n+1}=-i\left(  2\sqrt{2}x+n+2yi+\frac{\sqrt{2}-1}%
{2}\right)  ,
\end{array}
\right.
\]
we get a solution $\left(  2\sqrt{2}x+n\right)  y+\frac{\sqrt{2}-1}{2}y.$ The
proof is completed.
\end{proof}

We now define $\ $%
\begin{align*}
\left(  F_{1}^{\ast}\phi\right)  _{n}  &  =\partial_{s}\phi_{n}\theta_{n}%
-\phi_{n}\partial_{s}\theta_{n}-\lambda^{-1}\left(  \phi_{n+1}\theta
_{n-1}-\phi_{n}\theta_{n}\right)  ,\\
\left(  M_{1}^{\ast}\phi\right)  _{n}  &  =\partial_{t}\phi_{n+1}\theta
_{n}-\phi_{n+1}\partial_{t}\theta_{n}+\lambda\left(  \phi_{n}\theta_{n+1}%
-\phi_{n+1}\theta_{n}\right)  .
\end{align*}

\begin{lemma}
\label{L4}Suppose $\left\{  \phi_{n}\right\}  ,\left\{  \eta_{n}\right\}  $
satisfy $\left(  \ref{l2}\right)  .$ Then
\[
\partial_{x}\eta_{n}+\left(  2-\frac{\lambda}{\omega_{n}}-\frac{\lambda^{-1}%
}{\omega_{n+1}}\right)  \eta_{n}+\frac{\omega_{n+1}\lambda^{-1}}{\omega_{n}%
}\eta_{n-1}-\frac{\omega_{n}\lambda}{\omega_{n+1}}\eta_{n+1}=\frac{\left(
F_{1}^{\ast}\phi\right)  _{n}}{\omega_{n}}+\frac{\left(  M_{1}^{\ast}%
\phi\right)  _{n}}{\omega_{n+1}},
\]
and
\[
\frac{1}{i}\partial_{y}\eta_{n}+\left(  -\frac{\lambda}{\omega_{n}}%
+\frac{\lambda^{-1}}{\omega_{n+1}}-2\sqrt{2}\right)  \eta_{n}+\frac
{\omega_{n+1}\lambda^{-1}}{\omega_{n}}\eta_{n-1}+\frac{\omega_{n}\lambda
}{\omega_{n+1}}\eta_{n+1}=\frac{\left(  F_{1}^{\ast}\phi\right)  _{n}}%
{\omega_{n}}-\frac{\left(  M_{1}^{\ast}\phi\right)  _{n}}{\omega_{n+1}}.
\]
In particular, if $\left(  F_{1}^{\ast}\phi\right)  _{n}=\left(  M_{1}^{\ast
}\phi\right)  _{n}=0,$
\[
\eta_{n+1}\left(  x,y\right)  =\eta_{n}\left(  x+\frac{1}{2\sqrt{2}},y\right)
,
\]
and
\[
\left\vert \eta_{n}\right\vert \leq C\sqrt{1+x^{2}+y^{2}},
\]
then
\[
\eta_{n}=c_{1}\left(  2\sqrt{2}x+n+2yi\right)  .
\]

\end{lemma}

\begin{proof}
If $\left(  F_{1}^{\ast}\phi\right)  _{n}=\left(  M_{1}^{\ast}\phi\right)
_{n}=0,$ then
\[
\partial_{x}\eta_{n}+\left(  2-\frac{\lambda}{\omega_{n}}-\frac{\lambda^{-1}%
}{\omega_{n+1}}\right)  \eta_{n}+\frac{\omega_{n+1}\lambda^{-1}}{\omega_{n}%
}\eta_{n-1}-\frac{\omega_{n}\lambda}{\omega_{n+1}}\eta_{n+1}=0.
\]
Taking Fourier transform, we get
\begin{align*}
&  \left(  2\pi i\xi+2+\lambda^{-1}e^{-\frac{\pi i\xi}{\sqrt{2}}}-\lambda
e^{\frac{\pi i\xi}{\sqrt{2}}}\right)  \hat{\eta}_{n}\\
&  +\left(  1-e^{\frac{\pi i}{\sqrt{2}}\xi}\right)  \frac{\pi i}{\sqrt{2}%
}e^{2\pi i\left(  \alpha+\beta i\right)  \xi}\int_{-\infty}^{\xi}e^{-2\pi
i\left(  \alpha+\beta i\right)  s}\left(  \lambda-\lambda^{-1}e^{-\frac{\pi
i}{\sqrt{2}}s}\right)  \hat{\eta}_{n}\left(  s\right)  ds\\
&  =0,\text{ if }y>0,
\end{align*}
and
\begin{align*}
&  \left(  2\pi i\xi+2+\lambda^{-1}e^{-\frac{\pi i\xi}{\sqrt{2}}}-\lambda
e^{\frac{\pi i\xi}{\sqrt{2}}}\right)  \hat{\eta}_{n}\\
&  +\left(  1-e^{\frac{\pi i}{\sqrt{2}}\xi}\right)  \frac{\pi i}{\sqrt{2}%
}e^{2\pi i\left(  \alpha-\beta i\right)  \xi}\int_{+\infty}^{\xi}e^{-2\pi
i\left(  \alpha-\beta i\right)  s}\left(  \lambda-\lambda^{-1}e^{-\frac{\pi
i}{\sqrt{2}}s}\right)  \hat{\eta}_{n}\left(  s\right)  ds\\
&  =0,\text{ if }y<0.
\end{align*}
Then using the growth estimate of $\eta_{n},$ we find that $\eta_{n}%
=c_{1}\left(  2\sqrt{2}x+n+2yi\right)  .$
\end{proof}

We are in a position to prove our main theorem.

\begin{proof}
[Proof of Theorem \ref{Main}]Let $\left\{  \eta_{n}\right\}  $ be the solution
given by Lemma \ref{linear}. If $G_{1}\eta=N_{1}\eta=0,$ then by Lemma
\ref{L4}, $\eta_{n}=c_{1}\left(  2\sqrt{2}x+n+2yi\right)  .$ This implies that
$\eta_{n}=a_{1}\partial_{x}\theta_{n}+a_{2}\partial_{y}\theta_{n},$ for some
constants $a_{1},a_{2}.$

Now suppose $G_{1}\eta\neq0$ or $N_{1}\eta\neq0.$ By Propositon \ref{P2},
there exists $\left\{  \phi_{n}\right\}  $ such that $F_{1}\phi=G_{1}\eta,$
$M_{1}\phi=N_{1}\eta,$ and $\left\vert \phi_{n}\right\vert \leq C\left(
1+x^{2}+y^{2}\right)  ^{\frac{5}{8}}.$ Moreover, $T_{\omega}\phi=0.$

\medskip

\noindent Case 1. $G_{0}\phi=N_{0}\phi=0.$

In this case, by Lemma \ref{L3}, $\phi_{n}=c_{1}+c_{2}\omega_{n}.$ From this,
we deduce $\eta_{n}=a_{1}\partial_{x}\theta_{n}+a_{2}\partial_{y}\theta_{n}.$

\medskip

\noindent Case 2. $G_{0}\phi\neq0,$ or $N_{0}\phi\neq0.$

In this case, by Proposition \ref{P1}, we can find $\left\{  \sigma
_{n}\right\}  ,$ such that
\[
F_{0}\sigma=G_{0}\phi,\text{ }M_{0}\sigma=N_{1}\phi.
\]
Moreover, $\left\vert \sigma_{0}\right\vert \leq C\left(  1+x^{2}%
+y^{2}\right)  ^{\frac{5}{8}},$ and $T_{\kappa}\sigma=0.$ Taking Fourier
transform in the equation $T_{\kappa}\sigma=0,$ we conclude that
\[
\sigma_{n}=c_{1}+c_{2}\left(  2\sqrt{2}x+n\right)  +c_{3}y,
\]
for some constants $c_{1},c_{2},c_{3}.$ However, in view of Lemma \ref{M1},
after some computations, we find that $\eta$ can not satisfy the growth
control
\[
\left\vert \partial_{x}\eta_{0}\right\vert +\left\vert \partial_{y}\eta
_{0}\right\vert \leq C\frac{1}{\sqrt{1+x^{2}+y^{2}}}.
\]
Hence this case is also excluded.  This finishes the proof.

$\ $
\end{proof}

\end{document}